\documentclass[12pt,a4paper]{article}
\usepackage{amsmath,amsfonts,amsthm,amssymb}
\usepackage{graphicx}
\usepackage{verbatim}
\usepackage{color}
\usepackage{amscd}
\usepackage{verbatim}
\usepackage{mathrsfs}

\newtheorem{theorem}{\bf Theorem}[section]
\newtheorem{lemma}[theorem]{\bf Lemma}
\newtheorem{corollary}[theorem]{\bf Corollary}
\newtheorem{definition}[theorem]{\bf Definition}
\newtheorem{proposition}[theorem]{\bf Proposition}
\theoremstyle{remark}
\newtheorem{remark}[theorem]{\bf Remark}
\newtheorem{example}[theorem]{\bf Example}

\newcommand{\mysection}[1]{\section{#1}\setcounter{equation}{0}}
\def \OT  {{\mathbb{R}^{N}\times \R_T}} %{{\Omega_T}}

\def \d {{\delta}}

\def \K {\mathscr{K}}
\def \L {\mathscr{L}}
\def \M {\mathscr{M}}
\def \H {\mathscr{H}}
\def \A {\mathscr{A}}

\def \CC {{\mathscr{C}}}

\def \R {{\mathbb {R}}}
\def \Q {{\mathbb {Q}}}
\def \N {{\mathbb {N}}}

\def \Z {{\mathbb {Z}}}

\def \x {{\xi}}
\def \e {{\varepsilon}}

\def \t {{\tau}}

\def \y {{\eta}}

\def \g {{\gamma}}
\def \O {{\Omega}} %{\mathcal{O}}

\def \ex {\mathrm{ex}\,}
\def \exr {\mathrm{exr}\,}
\def \tilde {\widetilde}

\def\p{\partial}

\def\Delta{\varDelta}

\def \RR {\mathscr{R}}
\def\ci {{\mathbb {G_{\text C}}}} %{\mathfrak{C}}

\usepackage[english]{babel}
\begin{document}
\title{On Liouville-type theorems and the uniqueness of the positive Cauchy problem for a class of hypoelliptic
operators}
\author{{\sc{Alessia E. Kogoj\thanks{Dipartimento di Ingegneria dell'Informazione, Ingegneria Elettrica e Matematica Applicata, Universit\`a degli Studi di Salerno, IT-84084 Fisciano (SA),  (Italy).
E-mail: akogoj@unisa.it},
Yehuda Pinchover\thanks{Department of Mathematics, Technion - Israel Institute of Technology, Haifa 32000, (Israel).
\quad E-mail: pincho@tx.technion.ac.il},
 Sergio Polidoro\thanks{Dipartimento di Scienze Fisiche, Informatiche e Matematiche, Universit\`{a} di Modena e Reggio
Emilia, via Campi 213/b, 41125 Modena (Italy).
E-mail: sergio.polidoro@unimore.it}}}
}
\date{ }
\maketitle

\begin{abstract}
This note contains a representation formula for positive solutions of linear degenerate second-order equations of
the form
$$
  \partial_t u (x,t) = \sum_{j=1}^m X_j^2 u(x,t) + X_0 u(x,t) \qquad (x,t) \in \mathbb{R}^N \times\, ]- \infty ,T[,
$$
proved by a functional analytic approach based on Choquet theory. As a consequence, we obtain Liouville-type theorems
and uniqueness results for the positive Cauchy problem.
\\[2mm]
\noindent  2000  \! {\em Mathematics  Subject  Classification.}
Primary  \! 35K70; Secondary  35B09, 35B53, 35K15, 35K65.\\[1mm]
\end{abstract}

%35B53: Liouville theorems, Phragmén-Lindelöf theorems
%35H10: Hypoelliptic equations
%35K70: Ultraparabolic equations, pseudoparabolic equations, etc.
%35B45: A priori estimates
%35J25: Boundary value problems for second-order elliptic equations
%35B50: Maximum principles
%31B05: Harmonic, subharmonic, superharmonic functions
%35C15: Integral representations of solutions
%31B10: Integral representations, integral operators, integral equations methods
%35J70: Degenerate elliptic equations
%35H20: subelliptic equations
%32A19: Normal families of functions, mappings
%35R03: Partial differential equations on Heisenberg groups, Lie groups, Carnot groups, etc.
%35B08  	Entire solutions
%35B09  	Positive solutions
%35K15  	Initial value problems for second-order parabolic equations
%35K65  	Degenerate parabolic equations

%%%%%%%%%%%%%%%%%%
\mysection{Introduction}
%%%%%%%%%%%%%%%%%%%%%%%%%%%
In this article we consider second-order partial differential operators of the form
\begin{equation}\label{e1}
     \L u : = \p_t u -\sum_{j=1}^m X_j^2 u  - X_0 u \qquad \mbox{ in } \R^{N+1}.
\end{equation}
Points $z \in \R^{N+1}$ are denoted by $z=(x,t)$, where $x\in \R^{N}, t\in\R$. For $j=0,\ldots, m$, the $X_{j}$ are
vector fields
which are given by first-order linear partial differential operators in $\R^{N}$ with smooth coefficients
\begin{equation*}
     X_{j}(x):=\sum_{k=1}^{N} b_{jk}(x)\p_{x_{k}} \qquad j=0,\ldots, m.
\end{equation*}
We denote by $Y$ the {\it drift}
\begin{equation}\label{eY}
 Y := X_{0}-\p_{t}.
\end{equation}
We recall that the class of operators of the form  \eqref{e1} has been studied by many authors. In
particular, we refer to the monographs \cite{ LibroBLU, Bramantibook, Calin}, and to the references therein.

The aim of the article is to prove a representation formula for nonnegative solutions of $\L u= 0$ in the set
\begin{equation}\label{eOT}
  \OT :=  \R^N \times\, ]\!-\infty,T[,
\end{equation}
where $0 < T \le + \infty$. In the sequel we use the following notation
\begin{align}
  & \H := \Big\{u \in C^\infty (\OT) \mid \L u = 0 \quad \mbox{in } \OT \Big\}, \label{e-def-H} \\
  & \H_+ := \Big\{u \in \H \mid u \ge 0 \Big\}. \label{e-def-H+}
\end{align}
We use a functional analytic approach based on Choquet theory that allows us to represent all functions belonging to the
convex cone $\H_+$ in terms of its extremal rays. Moreover, we prove a \emph{separation principle} for the extremal
rays. The separation principle, in the nondegenerate case, says that (under certain conditions) nonnegative extremal
solutions of the heat equations have the form $u(x,t) = e^{\beta t} u_\beta(x)$, with $\beta \in \R$. In our degenerate
setting the separation principle has a different form that depends on $\L$. However, we prove in
Theorem~\ref{H*-lambda-repr-par} that, under some additional assumptions, any nonnegative extremal solution of
$\partial_t u = \sum_{j=1}^m X_j^2 u$ in $\OT$, does not depend on the `degenerate' variables. From the
representation theorem it plainly follows that under the additional assumptions also any function in $\H_+$ does not
depend on the `degenerate' variables. A similar result is proved in Theorem~\ref{H*-lambda-repr} for degenerate
stationary operators $\sum_{j=1}^m X_j^2 u = 0$, and in Corollary \ref{c_Kolmo} for Kolmogorov equations. We refer to
this kind of results as \emph{Liouville-type theorems} because of the very specific form of any point in $\H_+$.

Let us informally explain this remarkable phenomenon. We assume in Theorems~\ref{H*-lambda-repr-par} and
\ref{H*-lambda-repr} that $\L$ is invariant with respect to the \emph{left translations} of a nilpotent stratified Lie
group. On the other hand, the proof of our separation principle relies on Harnack inequalities that are invariant
with respect to the \emph{right translations} of the group. Both these two properties are satisfied in the particular
case of the last layer of the nilpotent Lie group. In this case, we can prove our separation principle, that yields our
claim. Let us also note that this fact is not completely unexpected. Indeed, Danielli, Garofalo and Petrosyan consider
in \cite{DGP2007} the  subelliptic obstacle problem in Carnot groups of step two, and prove that the non-horizontal
derivatives of any solution vanish continuously on the free boundary.

We also give a simple proof of a known uniqueness result for the positive Cauchy problem. We note that this integral
representation theory approach was previously used to prove the uniqueness of the positive Cauchy problem and
Liouville-type theorems for locally {\em uniformly} parabolic and elliptic operators \cite[and references
therein]{KoranyiTaylor85, LinPinchover94, Murata93, Murata95, Pinchover88, Pinchover1996}.

We next focus on Mumford and degenerate Kolmogorov operators. Their drift term $X_0$ is nontrivial, and
plays a crucial role in the regularity properties of the solutions. In Section \ref{sec_mumford} we prove a
uniqueness result for the positive Cauchy problem for Mumford operators. In Section \ref{sec_Kolmogorov} we consider a
family of degenerate Kolmogorov operators, and prove in Corollary \ref{c_Kolmo} that any nonnegative solution
of this partial differential equation in $\OT$ does not depend on the `degenerate' variables, and hence, the uniqueness of the positive Cauchy problem holds true for such operators.

\medskip

We list below our assumptions on $\L$ that will be used to accomplish this project. We assume that $\L$ satisfies the
celebrated H\"ormander's condition:
\begin{description}
  \item[{\rm (H0)}]
$ \qquad \text{rank Lie}\{X_{1},\dots,X_{m},Y\}(z) = N+1 \quad \text{for every} \, z \in \R^{N+1}.$
\end{description}
Under this condition H\"ormander proved in \cite{Hormander} that $\L$ is hypoelliptic, that is, any
distributional solution $u$ of the equation $\L u= f$ is a smooth classical solution, whenever $f$ is smooth.
In particular, $\H$ contains {\em all} distributional solutions of the equation $\L u=0$ in $\OT$.

Our second hypothesis is as follows:
\begin{description}
  \item[{\rm (H1)}] there exists a Lie group $\mathbb{G} = \left(\R^{N+1},\circ \right)$ such that the vector fields
$X_{1}, \dots, X_{m}, Y$ are invariant with respect to the left translation of $\mathbb{G}$. That is, for every $z,
\zeta \in \R^{N+1}$ we have
      \begin{equation*}
\begin{split}
  \left( X_j u \right) (\zeta \circ z) & = X_j \left( u(\zeta \circ z)\right) \qquad j = 1, \dots, m, \text{ and } \\
  \left( Y u \right) (\zeta \circ  z) & =  Y\left( u(\zeta \circ z)\right).
\end{split}
\end{equation*}
\end{description}
In particular, it follows from (H1) that
\begin{equation}\label{e-Lie}
  \big(\L u\big) (z) = f(z)\quad \Leftrightarrow \quad \L \big( u(\zeta \circ z) \big) = f(\zeta \circ z) \qquad
  \forall \zeta \in \R^{N+1}.
\end{equation}

\medskip

We will use the following notation in our further assumptions. As usual, we identify the first order linear partial
differential operator $X_j$ with the vector-valued function
\begin{equation*}
     X_{j}(x)= (b_{j1}(x), \ldots, b_{jN}(x)) \qquad j = 1, \dots, m.
\end{equation*}
For any $z_0 \in \R^{N+1}$ and any piecewise constant function $\omega: \left[0,T_0\right] \to \R^m$, let $\gamma$ be a
solution of the following initial value problem
\begin{equation}\label{e-gdot}
   \g'(s)=\sum_{j=1}^{m} \omega_j(s) X_j(\g(s))+ Y(\g(s)), \qquad \gamma(0)= z_0.
\end{equation}
We say that the solution $\g$ to \eqref{e-gdot} is an {\it $\L$-admissible path}.

Let $\O \subseteq\R^{N+1}$ be an open set and let $z_0 \in \O$. The {\it attainable set}
\begin{equation}\label{e-Anew}
   \A_{z_0} (\Omega):= \overline {A_{z_0}(\Omega)}
\end{equation}
is defined as the closure in $\O$ of
\begin{equation*}
 A_{z_0} (\Omega) := \big\{z\in\O \mid \;\exists \;\L\text{-admissible path } \g: [0,\tau] \to \O  \text{ s.t. } \g(0)=
z_0, \gamma(\tau)=z  \big\}.
\end{equation*}
When $\Omega=\OT$ (see \eqref{eOT}), we use the simplified notation $\A_{z_0} : = \A_{z_0} (\OT)$.

Our last requirement is concerned with a $\L$-admissible path with a constant $\omega\in \R^m$. As we will see in the sequel,
it yields a \emph{restricted uniform Harnack inequality} suitably modeled on the Lie group structure of $\mathbb{G}$
(cf. \cite{Murata93}). For $X=(X_1,\ldots, X_m)$, and $\omega = (\omega_{1},\dots,\omega_{m})\in \R^m$, we denote
\begin{equation*}
\begin{split}
  &\omega \cdot X := \omega_1 X_1 + \dots + \omega_m X_m, \\
  &  \exp\left( s \left(\omega \cdot X + Y \right) \right) z_0:=\gamma(s), \qquad \mbox{where is defined in
\eqref{e-gdot}}  \\
  & \exp\left( s \left(\omega \cdot X + Y \right)  \right) := \exp\left( s \left(\omega \cdot X + Y \right)
\right)(0,0).
\end{split}
\end{equation*}
Note that, by the invariance of the vector fields with respect to $\mathbb{G}$, we have
\begin{equation} \label{e-gcirc}
    \exp\left( s \left(\omega \cdot X + Y \right)  \right) z_0 =
    z_0 \circ\exp\left( s \left(\omega \cdot X + Y \right) \right).
\end{equation}
 Moreover, from \eqref{eY} we see that the \emph{time} component of $\exp\left( s \left(\omega \cdot X + Y \right)
\right) (x_0,t_0)$ is always $t_0-s$. With these notations, our last hypothesis reads as follows
\begin{description}
  \item[{\rm (H2)}] There exists a bounded open set $\Omega$ containing the origin, a vector $\omega
\in \R^m$ and a positive $s_0$ such that
\begin{equation}\label{eq-harnack}
    \exp\left( s \left(\omega \cdot X + Y \right) \right) \in \mathrm{Int} \left(\A_{(0,0)} (\Omega) \right)
\quad \text{for
any} \quad s \in\, ]0, s_0].
\end{equation}
\end{description}

\begin{remark} \label{rem-assumptions} Some comments on our assumptions (H0), (H1) and (H2) are worth noting.

1. The heat operator $\L = \partial_t - \varDelta$ is of the form \eqref{e1}. Moreover, it is invariant with
respect to the Euclidean translations $(x,t) \circ (\xi, \tau) = (x + \xi, t + \tau)$ and (H2) is satisfied by any
$\omega \in \R^N$. In this particular case, if we choose $\omega = 0$, and we recall that $X_0=0$, we see that
$\exp\left( s \left(\omega \cdot X + Y \right) \right) = (0,- s)$. Note that a restricted uniform Harnack inequality
$u(x,t-\e) \le C_\e u(x,t)$ follows from the classical parabolic Harnack inequality first proved by Hadamard
\cite{Hadamard} and Pini \cite{Pini}.

2. More generally, hypothesis (H2) is satisfied in the case of an operator of the form $\partial_t + \L_0 u$ in $\R^{N}
\times\,]\!-\infty, T[$, where $\L_0$ is a time-independent locally uniformly elliptic operator with bounded
coefficients, and also in the case of a manifold $M$ with a cocompact group action $G$ and an operator of the form
$\partial_t + \L_0 u$ on $M\times\, ]\!-\infty, T[$, where $\L_0$ is a (time-independent) $G$-invariant elliptic operator on $M$ (see
\cite{KoranyiTaylor85, LinPinchover94, Murata93, Pinchover88, Pinchover1996}).

3. We further note that there are operators $\L$ of the form \eqref{e1} that satisfy (H0) and (H1), for which  (H2) is
not satisfied for all $\omega$. We refer to Mumford operator \eqref{ex-Mum} discussed in Section
\ref{sec_mumford}, and to Example~\ref{ex4}.
\end{remark}

\medskip

Our assumptions (H0), (H1) and (H2) provide us with some compactness properties that are needed for proving that all points in the convex closed cone $\H_+$ can be represented in terms of its extremal rays. These compactness properties hinge on the following
local Harnack inequality which holds true under our assumptions (see the main result of \cite{KogojPolidoro15}).
\begin{description}
\item[{\rm (H*)}] Let $\O \subseteq\R^{N+1}$ be a bounded open set and let $z_0 \in \O$. For any compact
set $K \subset
\mathrm{Int}\left(\A_{z_0}(\O) \right)$ there exists a positive constant $C_K$, only depending on $\O, K, z_0$ and
$\L$, such that
\begin{equation}\label{lHarnack}
     \sup_K u \le C_K \, u(z_0),
\end{equation}
for any nonnegative solution $u$ of $\L u = 0$ in $\Omega$.
\end{description}
Note that, from (H*) and from the hypoellipticity of $\L$ we have that $\H$ is a Fr\'{e}chet space with respect to the
topology of uniform convergence on compact sets. Moreover, in this topology, $\H_+$ is clearly a closed convex cone
in $\H$. We denote by $\exr \H_+$ the set of all extreme rays of $\H_+$.

We next discuss the validity of (H*). Recall that Krener's Theorem states that for any open set $\O \subseteq\R^{N+1}$ and  $z_0 \in \O$, the
interior of $\A_{z_0}(\Omega)$ is not empty whenever (H0) is satisfied (see \cite{Krener} or \cite[Theorem 8.1, p.
107]{AgrachevSachkov}). We note here, that for this reason, it is not clear to us whether there exists an operator $\L$
satisfying (H0) and (H1), but not satisfying (H2).

Properties (H1), (H2) and (H*) yield the following \emph{restricted uniform Harnack inequality} (cf. \cite{Murata93}).

\medskip

\begin{proposition}[restricted uniform Harnack inequality]\label{p-restr-harnack}
Let $\L$ be an operator of the form \eqref{e1}, satisfying {\rm (H0)}, {\rm (H1)}, and {\rm (H2)}. Let $\omega$,
$\Omega$ and $s_0$ be as in (H2). For any $s>0$ there exists a positive constant $C_{s} >0$ depending only on
$\omega$, $s$ and $\L$, such that for any nonnegative solution $u$ of $\L u = 0$ in $\OT$ we have
\begin{equation}\label{e-harnack}
  u\left( \exp\left( s \left(\omega \cdot X + Y \right) \right) z \right) \le C_{s} u(z)
  \qquad \forall z \in \OT.
\end{equation}
Moreover, if for $j= 1, \dots, k$, $\omega_j$ are as in (H2), and $s_j$ are any positive constants,  then
there exists a positive constant $C_\mathbf{s} >0$ (where $\mathbf{s} = (s_1, \dots, s_k)$) depending only on
$\omega_1, \dots, \omega_k, \mathbf{s}$ and $\L$, such that for any nonnegative solution $u$ of $\L u = 0$ in $\OT$ we
have
\begin{equation}\label{e-harnack-2}
  u\left( \exp\left( s_k \left(\omega_k \cdot X + Y \right) \right) \dots \exp\left( s_1 \left(\omega_1 \cdot X +
Y \right) \right) z \right) \le C_\mathbf{s} u(z) \qquad \forall z \in \OT.
\end{equation}
\end{proposition}

\begin{proof}
Let $u$ be a nonnegative solution $u$ of $\L u = 0$ in $\OT$.  For any $z\in \mathbb{G}$ the function $u^z(y):=u(z\circ
y)$ is a nonnegative solution of the equation $\L u = 0$.  Therefore, for every $s \in ]0, s_0]$, by the local Harnack
inequality (H*) and \eqref{e-gcirc}, we have
\begin{equation} \label{e-gcirc1}
\begin{split}
    u\left( \exp\left( s \left(\omega \cdot X + Y \right)  \right) z\right) = u\left(z \circ\exp\left( s \left(\omega
\cdot X + Y \right)  \right)\right)= \\
    u^z\left(\exp\left( s \left(\omega \cdot X + Y \right)  \right)\right)\leq C_su^z(0)=C_{s} u(z).
\end{split}
\end{equation}
This proves \eqref{e-harnack} if $s \in ]0, s_0]$. If $s > s_0$ we choose $\tilde s \in ]0, s_0]$ and $k \in
\N$ such that $s = k \tilde s$. By \eqref{e-gcirc} and \eqref{e-gcirc1} we find.
\begin{equation*} %\label{e-gcirc2}
\begin{split}
    u\left( \exp\left( k \tilde s \left(\omega \cdot X + Y \right) \right) z\right) & \leq
    C_{\tilde s} u\left( \exp\left( (k-1) \tilde s \left(\omega \cdot X + Y \right)  \right) z\right) \leq \\
    \dots & \leq  C_{\tilde s}^{k-1} u\left( \exp\left( \tilde s \left(\omega \cdot X + Y \right)  \right) z\right)
    C_{\tilde s}^k u\left(z\right).
\end{split}
\end{equation*}
This concludes the proof of \eqref{e-harnack}, with $C_s = C_{\tilde s}^k$.

The proof of \eqref{e-harnack-2} follows by the same argument.
\end{proof}

\begin{remark} \label{rem-s-s_0} In the proof of Proposition \ref{p-restr-harnack} we have constructed a
\emph{Harnack chain} based on the \emph{local} Harnack inequality (H*). For this reason, \eqref{e-harnack} and
\eqref{e-harnack-2} don't require the boundedness assumption of the open set $\Omega$ and of the interval $]0,s_0]$ in
Condition (H2). Hence, when we apply Proposition \ref{p-restr-harnack} in the sequel, we don't refer to $\Omega$ and
$s_0$.
\end{remark}

The following  theorem is a version of the \emph{separation principle} (see \cite{Murata93} and
\cite[Definition~2.2]{Pinchover1996}). We note that the restricted uniform Harnack inequality
(Proposition~\ref{p-restr-harnack}) is used in the proof of our separation principle to construct \emph{Harnack chains}
along the path $\g(s) = \exp\left( s \left(\omega \cdot X + Y \right) \right)(x_0,t_0)$.

\medskip

\begin{theorem}[Separation principle] \label{th-main}
Let $\L$ be an operator of the form \eqref{e1}, satisfying {\rm (H0)}, {\rm (H1)}, and {\rm (H2)}. Let $\omega$
be as in {\rm (H2)}, and suppose that for every $u\in \H_+$, and every positive $s$
\begin{equation}\label{eq_right_inv}
    (x,t) \mapsto u\left(\exp(s (\omega \cdot X +Y))(x,t)\right) \qquad \mbox{is a solution of }\; \L u = 0 \mbox{ in
}\OT.
\end{equation}
Then, for every $u \in \exr \H_+$, $u\neq 0$, there exists $\beta \in \R$ such that
\begin{equation}\label{eq_funct_eq}
    u\left(\exp(s (\omega \cdot X +Y))(x,t)\right) = \mathrm{e}^{- \beta s} u(x,t)
\end{equation}
for every $(x,t) \in \OT$ and for every $s > 0$.
In particular, for every $u \in \exr \H_+$ and $z_0 = (x_0,t_0)$ in $\OT$, if $u(z_0) >0$, then $u>0$ in a neighborhood
of the integral curve
\begin{equation}\label{eq_int_curv}
\gamma:=\big\{ \exp\left( s \left(\omega \cdot X + Y \right) \right) z_0 \mid s \in \;] t_0 - T, + \infty[ \big\}.
\end{equation}
\end{theorem}

We also have the following result, useful in the study of stratified Lie groups and the Mumford operator. It is
weaker than Theorem \ref{th-main} in that the {\em right-invariance} of solutions is not assumed to hold for every
positive $s$.

\begin{proposition} \label{prop-separation}
Let $\L$ be an operator of the form \eqref{e1}, satisfying {\rm (H0)}, {\rm (H1)}, and {\rm (H2)}. Let $\omega_j$ is as
in (H2) for $j= 1, \dots, k$, and suppose that there exists $\mathbf{s} = (s_1, \dots, s_k) \in (\R^+)^k$ such that
\begin{equation}\label{eq_right_inv-2}
    (x,t) \mapsto u\left(\exp\left( s_k \left(\omega_k \cdot X + Y \right) \right) \dots \exp\left( s_1 \left(\omega_1
\cdot X + Y \right) \right)(x,t)\right)
\end{equation}
is a solution of $\L u = 0$ in $\OT$ whenever $u\in \H_+$. Then, for every $u \in \exr \H_+$, $u\neq 0$,
there exists a positive constant $C = C(\mathbf{s}, \omega_1, \dots, \omega_k)$ such that
\begin{equation}\label{eq_funct_eq-2}
    u\left(\exp\left( s_k \left(\omega_k \cdot X + Y \right) \right) \dots \exp\left( s_1 \left(\omega_1 \cdot
X + Y \right) \right)(x,t)\right) = C u(x,t)
\end{equation}
for every $(x,t) \in \OT$.
\end{proposition}

We prove Theorem~\ref{th-main} and Proposition \ref{prop-separation} in the next subsection devoted to our functional
setting.

\begin{remark}\label{rem-leftrightinvariance} Assumption \eqref{eq_right_inv} of Theorem~\ref{th-main} appears to be
quite
strong. Indeed, since $\L$ is \emph{left-invariant} with respect to the operation ``$\circ$'', it follows that $(x,t)
\mapsto u((x_0, t_0) \circ (x,t))$ is a solution of $\L u = 0$ for every fixed $(x_0,t_0) \in \R^{N+1}$ and $u\in \H$.
On the other hand, \eqref{e-gcirc}, says that $u\left(\exp(s (\omega \cdot X +Y))(x,t)\right) = u\left((x,t) \circ
\exp(s (\omega \cdot X +Y))\right)$, and therefore, we also assume, in fact, a \emph{right-invariance} condition, with
respect to the point $\exp(s (\omega \cdot X +Y))$.

However, both conditions are satisfied by the class of linear degenerate operators such that $X_0 = 0$. In this case we
have
\begin{equation} \label{eq-heatkernel}
  \L u(x,t) = \partial_t u(x,t) - \sum_{j=1}^m X_j^2 u(x,t),
\end{equation}
and (H2) is satisfied for every $\omega \in \R^m$. In particular, for $\omega = 0$ and $s>0$,
\begin{equation*} %\label{eq-heatkernel}
  \exp(s (\omega \cdot X +Y))(x,t) = (x,t-s)
\end{equation*}
and $\L u(x,t-s) = 0$ in $\OT$ if $\L u(x,t) = 0$ in $\OT$.

In Section~\ref{sec_Parabolic} we discuss some classes of operators of the form \eqref{eq-heatkernel} satisfying {\rm
(H0)}, {\rm (H1)}, and {\rm (H2)}. In this case, Theorem~\ref{th-main} says that for any nonnegative extremal solution
$u$ of $\L u = 0$ in $\OT$ there exists $\beta\in \R$ such that for any $s>0$
\begin{equation}\label{eq_funct_eq-nodrift}
    u(x,t-s) = \mathrm{e}^{- \beta s} u(x,t)   \qquad \forall (x,t) \in \OT.
\end{equation}
Note that a separation principle also holds when the drift term has the form $X_0 = \sum_{j=1}^N b_j \partial_{x_j}$,
where $b = (b_1, \dots , b_N)$ is any constant vector. Indeed, if $u$ is a positive solution of $$\partial_t u =
\sum_{j=1}^m X_j^2 u + \sum_{j=1}^N b_j \partial_{x_j}u$$ then $v(x,t) := u(x - t b, t)$ is a solution of the analogous
equation $$\partial_t v = \sum_{j=1}^m X_j^2 v.$$  Then we can apply Theorem~\ref{th-main} to $v$ with $\omega = b$,
and
finally we obtain
\begin{equation*}
   u(x + s b,t-s) = \mathrm{e}^{- \beta s} u(x,t)   \qquad \forall (x,t) \in \OT.
\end{equation*}

In Section \ref{sec_mumford}, we present a remarkable example of an operator satisfying assumption
\eqref{eq_right_inv} of Theorem~\ref{th-main}, namely, the well-known Mumford operator:
\begin{equation*} %\label{ex-Mum}
    \mathscr{M} u := \p_{t}  u - \cos(x) \p_y u - \sin(x) \p_w u - \p_x^2 u  \qquad (x,y,w,t) \in \R^4,
\end{equation*}
an operator that is discussed in detail in Section~\ref{sec_mumford}. Clearly its drift $X_0 = \cos(x) \p_y + \sin(x)
\p_w$ is nontrivial. It is also worth noting that $\mathscr{M}$ satisfies the assumptions of Proposition
\ref{prop-separation}, with $s = 2 \pi$, but it doesn't satisfy the hypotheses of Theorem \ref{th-main}.
We also note that Section \ref{s-remark} contains some remarks on the validity of \eqref{eq_right_inv} for operators with nontrivial
drift.
\end{remark}

\medskip

% Theorem~\ref{th-main} implies the uniqueness of the positive Cauchy problem (see Theorem~\ref{th-main-2}).
% The uniqueness of the positive Cauchy problem for a class of left translation invariant hypoelliptic operators was
% proved by Chiara Cinti in \cite{Cinti09} under the additional hypothesis that the operator is {\em homogeneous} with
% respect to a group of dilations on the underlying Lie group. The method used in \cite{Cinti09} relies on some accurate
% upper and lower bounds of the fundamental solution of $\L$. We note that the lower bounds for the fundamental solution
% are usually obtained by constructing suitable \emph{Harnack chains}, as the ones used in the proof of
%Theorem~\ref{th-main}.
% On the other hand, in order to apply the method used in \cite{Cinti09},  the upper and lower bounds need
% to agree asymptotically. Hence, the Harnack chains need to be chosen in some optimal way. An advantage of our method
%is
% that it does not require such an optimization step. Actually, a priori bounds of the fundamental solution, and even
%its
% existence are not needed. We also note that the bibliography of \cite{Cinti09} contains an extensive discussion of
%known results on the
% uniqueness of the Cauchy problem.

\medskip

The outline of the paper is as follows. In Section~\ref{sec_funct} we introduce representation formulas that play a
crucial role in our study, and we give the proof of Theorem~\ref{th-main}.
% Section~\ref{sec_example} contains some examples of operators satisfying our assumptions (H0)-(H2).
In sections~\ref{sec_Parabolic}--\ref{sec_Cauchy} we study operators $\L$ such that the drift term $X_0$ vanishes
identically. In particular, in Section~\ref{sec_Elliptic} we study stationary solutions, Section~\ref{sec_Parabolic}
deals with solutions of the evolution equation, while Section~\ref{ssec_Liouville} discusses parabolic Liouville-type
theorems, and Section~\ref{sec_Cauchy} contains a uniqueness result for the positive Cauchy problem. In
Section \ref{sec_mumford} we prove a new uniqueness result for Mumford's operator. In Section~\ref{sec_Kolmogorov}
we compute the Martin boundary of Kolmogorov-Fokker-Planck operators in $\OT$. Finally, Section~\ref{sec_further} is
devoted to some concluding remarks concerning the results of the present paper and to a discussion of some open
problems.

%%%%%%%%%%%%%%%%%%%%%%%%%%%%

\mysection{Functional setting}\label{sec_funct}

In the present section we introduce some notations, and recall some known facts about convex cones in vector spaces.
The following definition plays a crucial role in our study. It leads to some compactness results that enable us to apply
Choquet's theory. We first introduce the following notation. If $z \in \R^{N+1}$ and $\Omega$ is a bounded open subset
of  $\R^{N+1}$, we set
\begin{equation} \label{Omega-z}
 \Omega_z = z \circ \Omega =\Big\{ z \circ \zeta \mid \zeta \in \Omega \Big\}.
\end{equation}

\begin{definition}\label{def_ref_path}{\em Let $\L$ be an operator satisfying (H2).
A sequence $\RR := \left( z_k \right)_{k \in \N}\subset \OT$ is said to be a {\it
reference set} for $\L$ in $\OT$, if}
\begin{equation*}
  \bigcup_{k=1}^{\infty}  \mathrm{Int}  \left(\A_{z_k} \left(\Omega_{z_k} \right)\right) = \OT,
\end{equation*}
{\em where $\Omega$ is the bounded open set satisfying (H2).}
\end{definition}

We next prove that, in our setting, a reference set always exists.

\begin{proposition} \label{prop-covering} If $\L$ satisfies assumptions {\rm (H0), (H1)} and {\rm (H2)}, then a
\emph{reference set} $\RR$ exists.
\end{proposition}

\begin{proof}
Let $\left( K_j \right)_{ \in \N}$ be a sequence of compact sets  such that
\begin{equation*}
  \bigcup_{j=1}^{\infty}  K_j = \OT.
\end{equation*}

We claim that for every $j \in \N$ there exist $k_j \in \N$ and $z_{j_1}, \dots, z_{j_{k_j}} \in \OT$ such that
\begin{equation} \label{eq-covering}
  K_j \subset \bigcup_{i=1}^{{k_j}}  \mathrm{Int}  \left(\A_{z_{j_i}} \left(\Omega_{z_{j_i}} \right)\right).
\end{equation}
In order to prove \eqref{eq-covering} we consider $\omega \in \R^m, s_0>0$ and $\Omega$ satisfying (H2).
For every $(\xi, \tau) \in K_j$ we choose $s \in]0, s_0]$ such that $s + \tau <T$, and we set
\begin{equation*}
     (x,t) = \exp\left( - s \left(\omega \cdot X + Y \right) \right) (\xi, \tau).
\end{equation*}
Since $t = s+\tau<T$, it follows that  $(x,t) \in \OT$. Moreover
\begin{equation*}
     \exp\left( s \left(\omega \cdot X + Y \right) \right) (x,t) = (\xi, \tau),
\end{equation*}
then, by (H1) and (H2) we have that $(\xi, \tau) \in \mathrm{Int}\A_{(x,t)}\left( \Omega_{(x,t)} \right)$.
Hence, \eqref{eq-covering}
follows from the compactness of $K_j$. Therefore, a reference set for $\L$ in $\OT$ is given by
\begin{equation*}
  \RR := \bigcup_{j=1}^{\infty}  \Big\{ z_{j_1}, \dots, z_{j_{k_j}} \Big\}.
\end{equation*}
\end{proof}

We equip $\H$ with the compact open topology, that is, the topology of uniform convergence on compact sets.

Let $\RR := \left( z_k \right)_{k \in \N}$ be a reference set for $\L$ in $\OT$, and let $a = \left( a_k \right)_{k \in
\N}$ be a strictly positive sequence. We set
\begin{align}
  % & \H := \Big\{u \in C^\infty (\OT) \mid \L u = 0 \quad \mbox{in } \OT \Big\}, \label{e-def-H} \\
  % &\H_+ := \Big\{u \in \H \mid u \ge 0 \Big\}, \label{e-def-H+} \\
  &\H_{a} := \bigg\{u \in \H_+ \mid \sum_{k=1}^{\infty}  a_k u(z_k)  \le 1 \bigg\}, \label{e-def-Ha} \\
  &\H_{a}^1 := \bigg\{u \in \H_+ \mid \sum_{k=1}^{\infty}  a_k u(z_k) = 1 \bigg\}. \label{e-def-Ha1}
\end{align}
\begin{lemma}
For any positive sequence $a=\left( a_k \right)_{k \in
\N}$, the convex set $\H_{a}$ is compact in $\H_+$.
\end{lemma}
\begin{proof}
By the hypoellipticity of $\L$, it is sufficient to show that $\H_{a}$ is
locally bounded on $\OT$. With this aim, we consider any compact set $K \subset \OT$. By (H2), and
Proposition~\ref{prop-covering}, there exist $w_1, \dots, w_k$ in $\OT$ such that
\begin{equation*}
    K \subset \mathrm{Int} \left(\A_{w_1} \left(\Omega_{w_1} \right) \right) \cup \dots \cup
    \mathrm{Int} \left(\A_{w_k} \left(\Omega_{w_k} \right)\right).
\end{equation*}
We claim that there exist $z_{n_1}, \dots, z_{n_k}$ in $\RR$ and $k$ compact sets $K_1, \dots, K_k$ such
that
\begin{equation} \label{eq-covering-K}
    K = K_1 \cup \dots \cup K_k, \quad \text{and} \quad K_j \subset \mathrm{Int} \left(\A_{z_{n_j}} \left(\tilde \Omega_{j}
\right) \right),
\end{equation}
for  $j=1, \dots, k$, where every $\tilde \Omega_{j}$ is a bounded open set containing $z_{n_j}$.

Indeed, let $K_j := K \cap \left(\overline{\mathrm{Int}\left(\A_{w_j} \left(\Omega_{w_j} \right)\right)}\right)$ for
$j=1, \ldots, k$. As in the proof of Proposition~\ref{prop-covering}, we take $z_{n_j} \in \RR$ such  that $w_j
\in \mathrm{Int}\left(\A_{z_{n_j}} \left(\Omega_{z_{n_j}} \right)\right)$. We then choose a bounded open set $\tilde
\Omega_j$ containing $\Omega_{z_{n_j}} \cup K_j$, and we have that $K_j \subset \mathrm{Int}\left(\A_{z_{n_j}} \left(\tilde
\Omega_{j} \right)\right)$ for $j=1,\ldots , k$. This proves \eqref{eq-covering-K}.

\medskip

As a consequence of \eqref{eq-covering-K}, the restricted uniform Harnack inequality \eqref{e-harnack} yields
\begin{equation*}
    \sup_K u \le C_K \, \max_{j=1, \dots, k} u \left( z_{n_j} \right),
\end{equation*}
for some positive constant $C_K$ depending only on $\L$ and $K$. On the other hand, from the definition of $\H_a$ it
follows that for any $u\in \H_{a}$, we clearly have that $u \left( z_j \right) \le \frac{1}{a_j}$. Consequently,
\begin{equation*}
    \sup_K u \le C_K \, \max_{j=1, \dots, k} \left\{\dfrac{1}{a_{n_j}}\right\}.
\end{equation*}
\end{proof}
Note that $\H_+$ is the union of the caps $\H_{a}$. Indeed, for every $u \in \H_+$, we easily see that $u \in \H_{a}$
where the sequence $a = \left( a_k \right)_{k \in \N}$ is defined as follows $a_k := \frac{b_k}{u(z_k) + 1}$ and
$\left( b_k \right)_{k \in \N}$ is any nonnegative sequence such that $\sum b_k \le 1$.

Thus, $\H_{a}$ is a metrizable \emph{cap} in $\H_+$ (\emph{i.e.} $\H_{a}$ is a compact convex set and $\H_+
\backslash \H_{a}$ is convex) and $\H$ is {\em well-capped} (\emph{i.e.} $\H_+$ is the union of the caps $\H_{a}$).
Furthermore, since $\H_+$
is a harmonic space in the sense of Bauer, it follows that $\H_a$ is a \emph{simplex} (see \cite{Becker,Choquet}).

\medskip

Let $\CC$ be a convex cone, we denote by $\exr \CC$ the set of all extreme rays of $\CC$. Analogously, if $K$ is a
convex set, we denote by $\ex K$ the set of the extreme points of $K$.

Since $\H_+$ is a proper cone (\emph{i.e.}  it contains no one-dimensional subspaces), we have
\begin{equation}\label{e-exH}
    \ex \H_{a} = \big\{ 0 \big\} \cup \big\{ \exr \H_+ \cap \H_{a}^1 \big\}.
\end{equation}

\medskip

We next prove Theorem~\ref{th-main} and Proposition~\ref{prop-separation}. The argument of the proof is standard, we
give here the details for reader's convenience.

\begin{proof}[Proof of Theorem~\ref{th-main}] Clearly,  $\H_+\neq \{0\}$ since $\mathbf{1}\in \H_+$. By the
Krein-Milman theorem and \eqref{e-exH},  it follows that $\exr \H_+$ contains a nontrivial ray. Consider any function
$u \in \exr \H_+$ such that $u \not = 0$, and let $\omega \in \R^m$ be as in Proposition~\ref{p-restr-harnack}.
We claim that, for every positive $s$, there exists a positive constant $\alpha_s$ such that
\begin{equation} \label{e-cone}
    u\left(\exp\left( s \left(\omega \cdot X + Y \right)(x,t) \right) \right) = \alpha_s u(x,t).
\end{equation}
Indeed, let
\begin{equation*}
    v_s(x,t) := C^{-1}_{s} u\left(\exp\left( s \left(\omega \cdot X + Y \right)(x,t) \right)\right),
\end{equation*}
and recall that by our hypothesis \eqref{eq_right_inv}, $v_s$ is a nonnegative solution of the equation $\L v_s = 0$ in
$\OT$.
Moreover, the restricted uniform Harnack inequality (Proposition~\ref{p-restr-harnack}) implies that $ v_s \leq u$.
Since $u \in \exr \H_+$, it follows that $v_s(z)= \nu_su(z)$ for all $z\in \OT$, where $\nu_s\geq 0$. If $\nu_s>0$, then
we obviously have \eqref{e-cone}. Suppose that $\nu_s=0$, then by applying the exponential map forward, it follows that
$u(x,t)=0$ for all $(x,t)\in \Omega_{T-s}$. This completes the proof if $T=\infty$. If $T<\infty$ we repeat the
argument for a vanishing sequence $(s_j)_{j \in \N}$ of positive  numbers.  This contradicts our assumption
that $u\neq 0$. Hence \eqref{e-cone} is proved.

In order to conclude the proof of \eqref{eq_funct_eq}, we note that for every $\omega \in \R^m$ satisfying
the assumption of Proposition~\ref{p-restr-harnack}, $z \in \OT, s>0$, and  any $k \in \N$ we
have
\begin{equation*}
    \exp\left( k s \left(\omega \cdot X + Y \right) \right)z = \underbrace{\exp\left( s(\omega \cdot X + Y) \right)
\circ \ldots \circ \exp\left( s( \omega \cdot X + Y) \right)}_{k \ \text{times}} z.
\end{equation*}
 By iterating \eqref{e-cone}, we then find
\begin{equation*}
    \alpha_{k s} u(x,t) = u\left(\exp\left( k s \left(\omega \cdot X + Y \right) \right)(x,t) \right) = \alpha_s^k
u(x,t).
\end{equation*}
Hence, $\alpha_{k} = \alpha_1^k$, and  $\alpha_{1/k} = \alpha_{1}^{1/k}$, for every $k \in \N$. Therefore,
$\alpha_{r} = \alpha_1^r$ for every $r \in \Q$. The conclusion of the proof thus follows from the continuity of $u$, by
setting $\beta := \log (\alpha_1)$.

For the proof of the last assertion of the theorem, take $z_0$ such that $u(z_0)>0$. Then by \eqref{eq_funct_eq} $u>0$
on the integral curve $\gamma$ given by \eqref{eq_int_curv}.
\end{proof}

\medskip

\begin{proof}[Proof of Proposition~\ref{prop-separation}] It is analogous to the proof of \eqref{e-cone}, which is
based only on the Harnack inequality and on the assumption concerning the (restricted) right-invariance of the
solutions in $\H_+$. We omit the details.
\end{proof}

\medskip

\begin{remark} \label{r-separation}
When considering the classical heat equation in $\OT$, or more generally when $X_0=0$, the separation principle
reads as follows (see \cite{KoranyiTaylor85,Murata93,Pinchover88} for the corresponding result in the nondegenerate case):

\emph{For any $u \in \exr \H_+$ there exists $\lambda\leq \lambda_0$ such that}
\begin{equation} \label{ex-sep}
    u\left( x, t \right) = \mathrm{e}^{- \lambda t} u(x,0) \qquad \forall (x,t) \in \OT,
\end{equation}
where $\lambda _0$ is the {\em generalized principal eigenvalue} of the operator $\L_0:=- \sum_{j=1}^m X_j^2$ defined by
\begin{equation}\label{lambda0}
  \lambda_0:=\sup \Big\{\lambda \in \R \mid \exists u_\lambda\gneqq 0\;
\mbox{ s.t. }\Big(- \sum\nolimits_{j=1}^m X_j^2 - \lambda \Big)u_\lambda =0 \mbox{ in } \R^N \Big\}.
\end{equation}
Moreover, using Choquet's theorem and the argument in the proof of \cite[Theorem~2.1]{Pinchover88}, \eqref{ex-sep}
implies that $u$  is a nontrivial extremal solution of the equation $\L w = 0$ in $\OT$ if and only if it is of the
form $u(x,t) := e^{\lambda t} u_\lambda(x)$,  where $\lambda\leq \lambda_0$ and $u_\lambda$ is a nonzero extremal
solution of the equation $\L_\lambda \phi = (- \sum_{j=1}^m X_j^2 -\lambda )\phi = 0$ in $\R^{N}$.
In particular, it follows that any nontrivial solution in $\H_+$ is strictly positive.

In fact, for the heat equation it is known (see for example \cite{Doob}) that any nonnegative extremal
 caloric function $u\neq 0$ in $\R^{N+1}$ or in $\R^{N}\times \R_-$ is of the form
\begin{equation*} %\label{eq_funct_heat}
    u(x,t) = \exp\left(\langle x, v \rangle + t \|v\|^2  \right),
\end{equation*}
where $v\in\R^N$ is a fixed vector.

When a drift term $X_0$ appears in the operator $\L$, \eqref{ex-sep} does not holds necessarily, even for nondegenerate
parabolic equations. Consider, for instance, the nondegenerate Ornstein-Uhlenbeck operator
\begin{equation}\label{e-OU}
     \L u : = \p_t u - \Delta u - \langle x, \nabla u \rangle = 0 \qquad \mbox{ in }{ \R}^{N} \times \R_T.
\end{equation}
Clearly, $\L$ is of the form \eqref{e1} with $X_j = \partial_{x_j}, j=1, \dots, N$, and $X_0 = \langle x, \nabla \rangle
\simeq x$. Moreover, $\L$ is invariant with respect to the following change of variable. Fix any $(y,s) \in
\R^{N+1}$, and set $v(x,t) := u(x + e^{-t} y, t+s)$. We have that $\L v = 0$ in $\R^{N+1}$, if and only if $\L u = 0$
in $\R^{N+1}$. Thus, the Ornstein-Uhlenbeck operator satisfies (H0), (H1) and (H2). Note that in this case, the
restricted Harnack inequality reads as
\begin{equation*} %\label{e-H-OU}
     u \left(e^s x, t-s \right) \le C_s u (x,t)\qquad \forall (x,t) \in \R^{N+1},
\end{equation*}
and that \eqref{eq_right_inv} does not hold for $y\neq 0$. On the other hand, the expression of a \emph{minimal}
solution of the equation in one space variable, given in \cite{CranstonOreyRosler,Pinchover1996}, is
\begin{equation*}
     u_\lambda(x,t) = \exp \left( \lambda^2 e^{2t} - \sqrt{2} \lambda x e^t \right),
\end{equation*}
where $\lambda\in\R$. Clearly, \eqref{ex-sep} does not hold for $u_\lambda$.
\end{remark}

\mysection{Degenerate equations without drift}\label{sec_Parabolic}

We first derive from (H*) a Harnack inequality for the operator $\L - \lambda$, where $\L$ is of the form \eqref{e1}
and $\lambda$ is a real constant. After that, we focus on operators $\L$ such that the drift term $X_0$ does not appear.
In particular, we prove a representation theorem for the extremal nonnegative solutions of $\L u = 0$ in $\OT$, when $X_0
= 0$ and
the Lie group on $\R^N$ is nilpotent and stratified.

\begin{proposition} \label{H*-lambda}
Let $\L$ be an operator of the form \eqref{e1} that satisfies {\rm (H0)}, {\rm (H1)}, and {\rm (H2)}.
Let $\O \subseteq\R^{N+1}$ be an open set and let $z_0 = (x_0,t_0) \in \O$. For any compact set $K \subset
\mathrm{Int} \left(
\A_{z_0}(\O) \right)$ and for every $\lambda \in \R$ there exists a positive constant $C_{K, \lambda}$, only depending
on $\O, K, z_0, \lambda$ and $\L$, such that
\begin{equation*} %\label{lHarnack-lambda}
     \sup_K u \le C_{K, \lambda} \, u(z_0),
\end{equation*}
for any nonnegative solution $u$ of the equation $\L w-\lambda w = 0$ in $\Omega$.
\end{proposition}

\begin{proof}
If $u$ is a nonnegative solution of $\L w - \lambda w = 0$ in $\Omega$, then $u_\lambda(x,t) := e^{-\lambda t}
u(x,t)$ is a nonnegative solution of $\L w = 0$ in $\Omega$. The claim then follows from (H*) with
$C_{K, \lambda} := C_{K} \max_{(x,t) \in K} e^{\lambda (t_0-t)}$.
\end{proof}

We next consider operators $\L$ such that the drift term $X_0$ does not appear. We will use the following notation
\begin{equation} \label{e110}
  \L_0 := - \sum_{j=1}^m X_j^2 \, , \qquad  \L_\lambda := \L_0 - \lambda.
\end{equation}
We consider the degenerate  \emph{elliptic} equation $\L_\lambda u = 0$ in $\R^N$ and its parabolic counterpart $\L
= \partial_t u + \L_\lambda u = 0$ in $\R^N \times ]0,T[$. In this case H\"ormander's condition (H0) is equivalent to:
\begin{description}
  \item[{\rm (H0')}]
$ \qquad \text{rank Lie}\{X_{1},\dots,X_{m}\}(x) = N \quad \text{for every} \, x \in \R^{N}.$
\end{description}
Moreover, (H1) is equivalent to:
\begin{description}
  \item[{\rm (H1')}] there exists a Lie group $\mathbb{G}_0 = \left(\R^{N},\cdot \right)$ such that the vector fields
$X_{1}, \dots, X_{m}$ are invariant with respect to the left translation of $\mathbb{G}_0$.
\end{description}
Indeed, as (H1') is satisfied, then a group $\mathbb{G} = \left(\R^{N+1},\circ \right)$ satisfying (H1) is defined by
$\mathbb{G}:=\mathbb{G}_0\times\R$, with the operation
\begin{equation}
  (x,t) \circ (y,s) := (x \cdot y, t+s) \qquad (x,t), (y, s) \in \R^{N+1}.
\end{equation}
Finally, Chow-Rashevskii theorem (see for example \cite{Montgomery})  implies that for any
open cylinder $\Omega = O \times I$, with $O \subseteq\R^{N}$ an open connected set, and  an
interval $I \subset \R$, we have for every $(x_0, t_0) \in \Omega$ that
\begin{equation} \label{eq-Prop}
  \A_{(x_0, t_0)} (\Omega) = \Omega \, \cap\, \{(x,t)\mid t\leq t_0\},
\end{equation}
whenever (H0') holds. Thus condition (H2) is satisfied with any $\omega \in \R^m$. In the sequel of the present section
we will always consider $\omega = 0$.

Based on Proposition \ref{H*-lambda}, we next prove a Harnack inequality for the operators $\L_\lambda$. We
refer to the monograph \cite{LibroBLU} and to the reference therein for an exhaustive bibliography on Harnack
inequalities for operators of the form $\L_0$.

% Alessia will find good references on Harnack inequality for Sub-Laplacians (for instance: Capogna, Danielli,
% Garofalo Comm. in PDEs, (1993), Citti, Garofalo, Lanconelli Am. J. Math. (1993))

\begin{proposition} \label{H*-lambda-stationary}
Let $\L_0$ be an operator of the form \eqref{e110}, satisfying {\rm (H0')}, and {\rm (H1')}, and let $\lambda$ be a
given constant. Let $O \subseteq\R^{N}$ be an open connected set and let $x_0 \in O$. For any compact set $H \subset O$
there exists a positive constant $C_{H, \lambda}$, only depending on $O, H, x_0, \lambda$ and $\L_0$, such that
\begin{equation*} %\label{lHarnack-lambda}
     \sup_H u \le C_{H, \lambda} \, u(x_0),
\end{equation*}
for any nonnegative solution $u$ of $\L_\lambda u = 0$ in $O$.
\end{proposition}

\begin{proof}
If $u$ is a nonnegative solution of $\L_\lambda w = 0$ in $O$, then the function $v(x,t) := u(x)$ is a nonnegative
solution
of $\partial_t w + \L_\lambda w = 0$ in $\Omega := O \times I$, where $I$ is any open interval of $\R$. We choose $I =
]-2,1[$, $z_0 := (x_0, 0)$ and $K := H \times \{ -1 \}$. Then Chow-Rashevskii theorem implies $\A_{z_0}(\O) = \O
\cap \{ t \le 0 \}$, thus $K \subset \mathrm{Int} \left( \A_{z_0}(\O) \right)$. We then apply Proposition
\ref{H*-lambda} to $v$, and we obtain the Harnack estimate for $u$.
\end{proof}

We consider now operators of the form $\L_0$, satisfying {\rm (H0')}, and {\rm (H1')} with the further property that
they are invariant with respect to a family of dilations. Specifically, we suppose that $\R^N$ can be split  as follows
$$
  \R^N=\R^{m}\times \R^{m_2}\times \cdots \times \R^{m_n}, \quad \mbox{ and denote }
  x= (x^{(m)}, x^{(m_2)} \ldots, x^{(m_n)})\in \R^N,
$$
where $n\geq 2$. We assume that there exists a group of dilations $D_r:\R^N \to \R^N$, defined for every $r>0$ as
follows
\begin{eqnarray*}
D_r (x)= D_r \left(x^{(m)}, x^{(m_2)}, \ldots, x^{(m_n)}\right):= \left(r x^{(m)}, r^2 x^{(m_2)}, \ldots, r^n
x^{(m_n)}\right),
\end{eqnarray*}
which are automorphisms of $(\R^N,\cdot)$. In this case we say that $\ci = \left(\R^N, \cdot, (D_r)_{r>0}\right)$ is a
{\em homogeneous Lie group}. It is well-known that $\ci$ is nilpotent (see for example \cite[Proposition~1.3.12]{LibroBLU})
and compactly generating. Moreover, the following two properties follow from the homogeneous structure of
Carnot groups (see for example \cite[Theorem~1.3.15]{LibroBLU}).
\begin{equation} \label{eq-firstlayer}
  (x \cdot y)^{(m)} = x^{(m)} + y^{(m)} \qquad \text{for every}  \quad x, y \in \R^N.
\end{equation}
% it follows that the operation of the group acts in the first $m$ variables as
% follows. If $x, y \in \R^N$ we have that $(x \cdot y)^{(m)} = x^{(m)} + y^{(m)}$ .
\begin{equation} \label{eq-lastlayer}
  x \cdot y = y \cdot x = x + y \qquad \text{whenever}  \quad x= (0^{(m)}, 0^{(m_2)} \ldots, 0^{(m_n)}, x^{(m_n)}).
\end{equation}
We point out that \eqref{eq-lastlayer} means, in particular, that right and left multiplications by a point $x$
belonging to the last layer of the group agree.

When the vector fields $X_1, \dots, X_m$ are homogeneous of degree $1$ with respect to the dilation $(D_r)_{r>0}$,
we say that $\ci:=\left(\R^N,\cdot,(D_r)_{r>0}\right)$ is a {\it Carnot group} and $X_1,\ldots, X_{m}$ are called {\it
generators} of $\ci$. In this case, it is always possible to choose the $X_j$'s such that $X_j = \partial_{x_j} +
\sum_{k=m+1}^{N} b_{jk}(x)\p_{x_{k}}$, for $j=0,\ldots, m$, and the coefficients $b_{jk}(x)$ are polynomials. Moreover
all commutators $[X_j, X_k]$ only acts on $(x^{(m_2)}, \ldots, x^{(m_n)})$, third order commutators  $[X_i, [X_j, X_k]]$
only acts on $(x^{(m_3)}, \ldots, x^{(m_n)})$, $n$-th order commutators only act on $x^{(m_n)}$.

The corresponding {\em sub-Laplacian} $\varDelta_{\mathbb{G}} = \sum_{j=1}^m X_j^2$ agrees with $- \L_0$, and is always
self-adjoint, that is $\varDelta_{\mathbb{G}}^* = \varDelta_{\mathbb{G}}$. For an extensive treatment on sub-Laplacians
on Carnot groups we refer to the book \cite{LibroBLU} by Bonfiglioli, Lanconelli and Uguzzoni.

\begin{example}\label{ex0-ell}  {\sc Heisenberg group.} \ $\mathbb{H} := \left(\R^3, \cdot \right)$, is defined by the
multiplication
\begin{equation*}
     (\xi,\eta, \zeta) \cdot (x,y,z) := \left( \xi + x, \eta + y, \zeta + z + (\eta x - \xi y) \right) \quad
(\xi,\eta, \zeta), (x,y,z) \in \R^3.
\end{equation*}
The vector fields $X_1$ and $X_2$
\begin{equation*}
     X_1 := \partial_{x} - \tfrac12 y \partial_{z}, \quad X_2 := \partial_{y} + \tfrac12 x \partial_{z},
\end{equation*}
are invariant with respect to the left translation of $\mathbb{H} = \left(\R^3, \cdot \right)$, and with respect to the
following dilation in $\R^3$
\begin{equation*}
     D_r (x,y,z) := \left( r x, r y, r^2 z \right)\qquad (x,y,z) \in \R^3, r > 0.
\end{equation*}
% Moreover $X_1^* = - X_1$, and $X_2^* = - X_2$, then its Sublaplacian $\varDelta_{\mathbb{H}} = X_1^2 + X_2^2$
% is self-adjoint.
Note that we have $[X_1, X_2] = \partial_{z}$, and any other commutator is zero.

The sub-Laplacian on the Heisenberg group acts on a function $u = u(x,y,z)$ as follows
\begin{equation} \label{e-HG}
  \varDelta_{\mathbb{H}} u:= \left(\partial_{x} - \tfrac12 y \partial_{z}\right)^2 u(x,y,z) + \left( \partial_{y} +
\tfrac12 x \partial_{z} \right)^2 u(x,y,z).
\end{equation}
\end{example}

\medskip

If $- \L_0$ is a sub-Laplacian in a Carnot group $\ci:=\left(\R^N,\cdot,(D_r)_{r>0}\right)$, we define
a homogeneous group $\mathbb{G} =\left(\R^{N+1},\circ,(\delta_r)_{r>0}\right)$
$$
    (x,t) \circ (\xi, \t) := ( x \cdot \xi, t + \t),\qquad \d_r (x,t) := \big( D_r x, r^2 t \big)
$$
for every $(x,t) (\xi, \tau) \in \R^{N+1}$ and for any $r>0$. For $\omega = 0$ we have $\exp\left( \t (\omega \cdot X
+ Y) \right) = (0, -\t)$. Thanks to the invariance with respect to translations and dilations, the restricted uniform
Harnack inequality of Proposition~\ref{p-restr-harnack} for such an operator $\L$ reads as
\begin{equation} \label{e-harnack-Heisenberg}
  u\left(x, t-\tau \right) \le C_{\tau}  u(x,t) \qquad \text{for every} \ (x,t) \in \OT, \tau >0,
\end{equation}
and for any nonnegative solution of $\L u = 0$ in $\OT$.

The main result of this Section is the following version of the separation principle.

\begin{theorem} \label{H*-lambda-repr-par}
Let $\ci = \left(\R^N, \cdot, (D_r)_{r>0}\right)$ be a Carnot group, let $\varDelta_{\mathbb{G}}$ be its sub-Laplacian,
and assume that $\L_0  = - \sum_{j=1}^m X_j^2$ agrees with $- \varDelta_{\mathbb{G}}$. If $u$ is an extremal
nonnegative solution of $\L u = \partial_t u - \varDelta_{\mathbb{G}} u = 0$ in $\OT$, then
\begin{equation*}
  u(x,t) = \exp \left( \langle x, \alpha \rangle  + |\alpha|^2 t \right)
\end{equation*}
for some vector $\alpha =( \alpha_1, \dots, \alpha_m, 0, \dots, 0)$. Moreover, any nontrivial solution $v\in \H_+$ does
not depend on the `degenerate' variables $x_{m+1},\ldots,x_N$, and $v$ is strictly positive.
\end{theorem}

\begin{proof} We first give the proof in the simplest (nontrivial) case of the Heisenberg group $\mathbb{H}$, in order
to show the main idea of the proof. Let $c$ be any real constant, and let $(x,y,z,t)$ be a given point of $\R^4$. A
direct computation shows that
\begin{multline} \label{eq_funct-H}
  \exp\!\big( s \!\left(\!- c X_2 \!-\! \partial_t \right)\! \big)\! \exp\!\big(s \!\left(\!- c X_1 \!-\! \partial_t \!\right)\!  \big)\!
  \exp\!\big(s \!\left(\!c X_2 \!-\! \partial_t \!\right)\! \big)\!  \exp\!\big(s \!\left(\! c X_1 \!-\! \partial_t \!\right)  \! \big)\!(x,y,z,t)=\\[2mm]
(x,y,z + c^2 s^ 2, t - 4 s),
\end{multline}
for every positive $s$. Note that for any $u\in\H_+$, we have that
$$
  v(x,y,z,t):=u(x,y,z + c^2 s^ 2, t - 4 s) \in \H_+.
$$
Since hypothesis (H2) holds true, Proposition \ref{prop-separation} implies that for any extremal solution $u\in\H_+$
there exists a positive constant $C_s$, that may depend on $c$, such that
\begin{equation*}
    u(x,y,z + c^2 s^ 2, t - 4 s) = C_s u(x,y,z, t)   \qquad \forall (x,y,z, t) \in \OT,
\end{equation*}
and for every positive $s$. The standard argument used in the last part of the proof of Theorem \ref{th-main} implies
that
\begin{equation}\label{eq_funct_H}
    u(x,y,z + c^2 s^ 2, t - 4 s) = e^{\beta_c s} u(x,y,z, t)   \qquad \forall (x,y,z, t) \in \OT,
\end{equation}
and for every positive $s$. Note that for $c=0$, the above identity restores \eqref{eq_funct_eq-nodrift}
\begin{equation*}
    u(x,y,z, t -s) = e^{\tilde\beta_0 s} u(x,y,z, t).
\end{equation*}
Combining it with \eqref{eq_funct_H} we find
\begin{equation*}
    u(x,y,z + c^2 s^ 2, t) = e^{\tilde\beta_c s} u(x,y,z, t)   \qquad \forall (x,y,z, t) \in \OT,
\end{equation*}
for some real constant $\tilde\beta_c$. The above identity can be written equivalently as
\begin{equation}\label{eq_funct_H-bis}
    u(x,y,z', t) = e^{\tilde\beta_c \sqrt{|z'-z|}} u(x,y,z, t)   \qquad \forall (x,y,z, t), (x,y,z', t) \in \OT.
\end{equation}
We finally note that \eqref{eq_funct_H-bis} contradicts the regularity of $u$ unless
$\tilde\beta_c=0$. Since $u$ is smooth by H\"ormander's condition (H0), we have necessarily $\tilde\beta_c=0$.
Hence $u = u(x,y,t)$ is a nonnegative extremal solution of $\partial_t u = \Delta u$, and the conclusion of the proof,
in the case of the Heisenberg group $\mathbb{H}$, follows from the classical representation theorem for the heat
equation \cite{Pinchover88}.

Before considering any Carnot Group $\ci$, we point out that the above proof only relies on the fact that $\partial_z$
is the highest order commutators of a nilpotent Lie group. In particular, the operator $\L$ is translation invariant
with respect to $z$ and that $\partial_z$ has been obtained by \eqref{eq_funct-H}. Then Proposition
\ref{prop-separation} gives \eqref{eq_funct_H-bis}, that in turns contradicts the smoothness of $u$.

Let $\L = \partial_t - \varDelta_{\mathbb{G}}$, where $\varDelta_{\mathbb{G}}$ is a sub-Laplacian on a Carnot
group $\ci$. We recall the Baker-Campbell-Hausdorff formula. If $X_j, X_k$ are the vector fields belonging to the first
layer of $\ci$, then
\begin{equation*} %\label{eq_funct_h-j}
\begin{split}
  \exp\big(s \left(X_j - \partial_t \right) \big) & \exp\big(s \left(X_k - \partial_t \right)  \big)(x,t) = \\
  &  \exp\left(s \left( \big( X_j + X_k \big) - 2 \partial_t\right)  + \tfrac{s^2}{2}\big[ X_k, X_j \big] + R_{jk}(s)
\right) (x,t)
\end{split}
\end{equation*}
for any $s \in \R$, where $R_{jk}$ denotes a polynomial function of the form
\begin{equation*} %\label{eq_funct_h-j}
  R_{jk}(s) = \sum_{i=3}^m c_{i,jk} s^i
\end{equation*}
whose coefficients $c_{i,jk}$'s are sums of commutators of $X_1, \dots, X_m$ of order $i$. In particular, we have
\begin{multline*} %\label{eq_funct_h-j}
  \exp\!\big(s \!\left(\!- X_j \!-\! \partial_t \!\right)\! \big)\! \exp\!\big(s \!\left(\!-X_k \!-\! \partial_t \!\right)\!  \big)\! \exp\!\big(\!s\! \left(\!X_j\! - \!\partial_t \!\right) \!\big)\!  \exp\!\big(\!s \!\left(\!X_k \!-\! \partial_t \!\right)\!  \big)\!(x,t)= \\[2mm]
    \exp\left(-4 s \partial_t  + \tfrac{s^2}{2}\big[ X_k, X_j \big] + R_{jk}(s) \right) (x,t).
\end{multline*}
We can express the variable $x^{(m_n)}$ of the last layer of $\ci$ in terms of commutators of order $n$ with zero
reminder. In particular, by repeating the use of the Baker-Campbell-Hausdorff formula, we can express every vector
$x_j^{(m_n)}$ of a basis of the last layer of $\ci$ as
\begin{equation*} %\label{eq_funct_h-j}
  x_j^{(m_n)} = \exp\big( - X_{j_k} \big) \dots \exp\big( - X_{j_1}),
\end{equation*}
for a suitable choice of $X_{j_1}, \dots, X_{j_k}$ in the fist layer of $\ci$. In particular, we have that
\begin{equation*} %\label{eq_funct_h-j}
 u\big(x + s^n x_j^{(m_n)},t-k s\big) = u \left(\exp\big( - s( X_{j_k} - \partial_t) \big) \dots \exp\big( - s (
X_{j_1}- \partial_t)\big) (x,t) \right),
\end{equation*}
for every $(x,t) \in \OT$ and every positive $s$. On the other hand, by \eqref{eq-lastlayer} $x + s
x_j^{(m_n)}$ is at once a \emph{right} and \emph{left} translation on the group $\ci$. Then, in particular, $(x,t)
\mapsto u \big(x + s x_j^{(m_n)},t\big)$ is a solution of $\L_0 u = 0$ for every $s \in \R$.
Thus, if $u$ is an extremal solution of $\L_0 u = 0$, Proposition
\ref{prop-separation}, combined with \eqref{eq_funct_eq-nodrift}, yields
\begin{equation*} %\label{eq_funct_h-j}
 u\big(x + c s^n x_j^{(m_n)},t\big) = e^{\beta s} u(x,t),
\end{equation*}
for every $x \in \R^N$ and $s \ge 0$. Here $c$ is a real constant that may depend on $x_j^{(m_n)}$. As in the
case of the Heisenberg group, this identity contradicts the smoothness of $u$, unless $u$ doesn't depend on $x^{(m_n)}$.
Thus, $u = u\left(x^{(m)}, \dots, x^{(m_{n-1})} \right)$ is an extremal solution of $\L' u = 0$, where $\L' = \partial_t
- \varDelta_{\mathbb{G}'}$, and $\varDelta_{\mathbb{G}'}$ is a sub-Laplacian on a Carnot group $\mathbb{G}'$ on $\R^{N -
m_n}$ defined as the restriction of $\mathbb{G}$ to the first $N-m_n$ variables of $\R^N$. The conclusion of the proof
follows by a backward iteration of the above argument.
\end{proof}

\mysection{Stationary equations} \label{sec_Elliptic}
In the present section we consider stationary equations, and we prove a result analogous to Theorem
\ref{H*-lambda-repr-par}. We first introduce some notations. Fix any $\lambda \in \R$, and consider an operator
$\L_\lambda$ of the form \eqref{e110} on $\R^N$, satisfying {\rm (H0')}, and {\rm (H1')}. We set
\begin{align}
  & \H_\lambda := \Big\{u \in C^\infty (\R^N) \mid \L_\lambda u = 0 \quad \mbox{in } \R^N \Big\}, \label{e-def-Hl} \\
  & \H_\lambda^+ := \Big\{u \in \H_\lambda \mid u \ge 0, u(0)=1 \Big\}. \label{e-def-Hl+}
\end{align}
Note that in light of Proposition~\ref{H*-lambda-stationary}, the generalized principal eigenvalue $\lambda_0$
defined in \eqref{lambda0} can be characterized as
\begin{equation*} %\label{lambda0}
  \lambda_0:=\sup \Big\{\lambda \in \R \mid \H_\lambda^+ \neq \emptyset \Big\}.
\end{equation*}
Moreover, by the strong minimum principle (or Proposition~\ref{H*-lambda-stationary}), any function $u \in
\H_\lambda^+$ never vanishes. The results proved in Section~\ref{sec_funct} for $\H$ and $\H_+$, and $\H_a$ plainly
extend to $\H_\lambda$ and $\H_\lambda^+$. In particular, it follows that $\H_\lambda^+$ is a convex compact set (for a
reference set for $\L_\lambda$ in $\R^N$ one can choose any singleton). Hence any function in $\H_\lambda^+$ can be
represented by the set of all extreme points of $\H_\lambda^+$.
%It follows that if $\mathbb{G}$ is nilpotent, then Theorem~6.8 of \cite{LinPinchover94} applies.

\begin{theorem} \label{H*-lambda-repr}
Let $\ci = \left(\R^N, \cdot, (D_r)_{r>0}\right)$ be a Carnot group, let $\varDelta_{\mathbb{G}}$ be its
sub-Laplacian, and assume that $\L_0  = - \sum_{j=1}^m X_j^2$ agrees with $- \varDelta_{\mathbb{G}}$.
Then $\lambda_0 = 0$, and for any $\lambda \leq 0$,  $u\in \H_\lambda^+$ is an extremal solution if and only if
\begin{equation*}
  u(x) = u_\alpha(x):=\exp \left( \langle x, \alpha \rangle \right)
\end{equation*}
for some vector $\alpha =( \alpha_1, \dots, \alpha_m, 0, \dots, 0)$ such that $\|\alpha\|^2 = - \lambda$. Moreover,
$u\in \H_\lambda^+$ if and only if there exists a unique probability measure $\mu$ on $\mathbb{S}^{m-1}$ such that
$$u(x)=\int_{\xi\in \mathbb{S}^{m-1}} \exp \left( \sqrt{-\lambda}\langle x, \xi \rangle \right) \,\mathrm{d}\mu(\xi).$$
\end{theorem}

\begin{proof} It is a direct consequence of Theorem \ref{H*-lambda-repr-par} and Choquet's theorem. Recall that as in
\cite[Theorem~2.1]{Pinchover88}, if the separation principle of the form \eqref{ex-sep} holds true, then $u_\lambda$ is
an extremal solution of $\L_\lambda v = 0$ in $\R^N$ if and only if the function $u(x,t) := e^{\lambda t} u_\lambda(x)$
is a nonzero extremal solution of $\L w = 0$ in $\R^{N+1}$ (see also, Remark~\eqref{r-separation}). The conclusion
immediately follows from Theorem \ref{H*-lambda-repr-par}.
\end{proof}

As a result we obtain the following nonnegative Liouville theorem.

\begin{corollary}\label{cor-Liouville}
If $u \in \H_0^+$ and $ - \L_0$ is a sub-Laplacian $\varDelta_{\mathbb{G}}$ on a Carnot group $\mathbb{G}$,
then $u =\mathbf{1}$, where $\mathbf{1}$ is the constant function taking the value $1$ in $\R^N$.
\end{corollary}

\begin{remark}
  If $\L_0 = - \sum_{jk}^m a_{jk} X_j X_k$ for some symmetric positive definite constant matrix $A = \left( a_{jk}
\right)_{j,k=1, m}$, then the result of Theorem~\ref{H*-lambda-repr} clearly applies with
\begin{equation*}
  u(x) = u_{A;\alpha}(x)=\exp \left( \langle A^{-1} x, \alpha \rangle \right),
\end{equation*}
with $\alpha=(\alpha_1, \dots, \alpha_m, 0, \dots, 0)\in \R^N$ such that $\langle A^{-1} \alpha, \alpha \rangle = -
\lambda$.
\end{remark}

\mysection{Parabolic Liouville theorems}\label{ssec_Liouville}
In the present section we assume that $\L$ is a hypoelliptic operator of the form
\begin{equation} \label{e11n}
  \L := \partial_t  + \L_0, \qquad \L_0 := - \sum_{j=1}^m X_j^2
\end{equation}
satisfying {\rm (H0')} and {\rm (H1')}. In particular, $\L$ is of the form \eqref{e1} with
$X_0=0$.

We say that $ \L_0$ satisfies the \emph{nonnegative Liouville property} if any nonnegative solution of $\L_0 u=0$ in
$\R^N$ is equal to a constant. Recall that
\begin{equation*} %\label{lambda0}
  \lambda_0:=\sup \Big\{\lambda \in \R \mid \exists u_\lambda\gneqq 0\;
\mbox{ s.t. }\left(\L_0 - \lambda \right)u_\lambda =0 \mbox{ in } \R^N \Big\}
\end{equation*}
denotes the generalized principal eigenvalue of the operator $\L_0$.

We assume that
\begin{description}
  \item[{\it a)}] $\L_0$ satisfies the nonnegative Liouville property,
  \item[{\it b)}] $\lambda_0=0$.
\end{description}
We note that the nonnegative Liouville property clearly implies the Liouville property for {\em bounded} solutions: any
bounded solution of $\L_0 u=0$ in $\R^N$ is equal to a constant.

Properties {\it (a)-(b)} hold whenever $\mathbb{G}$ is nilpotent and $\L_0=\L_0^*$ (see \cite{LinPinchover94} for a
similar statement), and in particular, under the assumptions of Theorem~\ref{H*-lambda-repr} (see the aforementioned
theorem and Corollary~\ref{cor-Liouville}, see also \cite{KogojLanconelli07}).
Property {\it (a)} also holds when all the $X_j$'s are homogeneous of degree 1 with respect to a
dilation group.  It is also true for a wide class of operators including Grushin-type operators
\begin{equation*}
  \L_{0} = - \partial_x^2 - x^{2 \alpha} \partial_y^2\,,
\end{equation*}
where $\alpha$ is any positive constant (see \cite{KogojLanconelli09}). Property {\it (b)} is well studied in
the nondegenerate case, and our Theorem \ref{H*-lambda-repr} is a first result for degenerate operators. We aim to
study this property under more general assumptions in a forthcoming work.

Since $X_0=0$, Theorem~\ref{th-main} implies that a nonzero  $u\in \mathrm{exr}\,\H^+$ if and only if it satisfies
the separation principle, namely, $$u(x,t)=e^{-\lambda t} \varphi_\lambda( x),$$
where $\varphi_\lambda$ is an extreme positive solution of the equation $\big(\sum_{j=1}^m X_j^2 +\lambda\big) u =0$ in
$\R^N$, and $\lambda\le \lambda_0=0$. Consequently, the following nonnegative Liouville theorem holds for $\L$ in
$\R^{N}\times \R$.

\begin{theorem}\label{thm_end}
Assume that $\L_0$ satisfies the nonnegative Liouville property and that $\lambda_0=0$.
Let $u\geq 0$ be a solution of the equation
$$
 \left(\partial_t + \L_0 \right) u = \partial_t u - \sum_{j=1}^m X_j^2 u = 0
 \qquad \text{in} \quad \R^{N}\times \R
$$
such that
$$
  u(0,t)=O(\mathrm{e}^{\varepsilon t})\qquad  \mbox{ as } t\to\infty,
$$
for any $\varepsilon>0$. Then $u=\mathrm{constant}$.
\end{theorem}

This result should be compared with the Liouville theorems proved by Kogoj and Lanconelli in
\cite{KogojLanconelli05,KogojLanconelli06,KogojLanconelli07},  where it was assumed that the operator $\L$ is of the
form \eqref{e1}, $\L$  is not necessarily translation invariant, but it is invariant with respect to a dilation group
$\left( \d_r \right)_{r > 0}$, and satisfies an \emph{oriented connectivity condition} that is, (using  our notation)
\begin{equation}
  \A_{(x_0,t_0)} = \R^N \times ]-\infty, t_0[, \qquad \text{for every} \quad (x_0,t_0) \in \OT.
\end{equation}
In this case,  a (stronger) sufficient growth condition for the validity of the above Liouville theorem is
$$u(0,t)=O(t^m)\qquad  \mbox{ as } t\to\infty, \mbox{ for some } m>0.$$
In particular, in this case, the nonnegative Liouville theorem holds true for the stationary equation (without any
growth condition, see \cite[Corollary 1.2]{KogojLanconelli05}).

\mysection{Positive Cauchy Problem} \label{sec_Cauchy}

In the present section we consider the positive Cauchy problem for $\L$ in $S_T := \R^N \times\, ]0,T[$ with $0 < T \le
+
\infty$, where $\L$ is of the form \eqref{e11n}. Our aim is to prove the following uniqueness result for the positive
Cauchy problem under the assumption that $X_0=0$.
\begin{theorem} \label{th-main-2}
{Let $\L$ be an operator of the form \eqref{e11n}, satisfying {\rm (H0')} and {\rm (H1')}, and let
$u_0 \geq 0$ be a continuous function in $\R^N$. Then the positive Cauchy problem
\begin{equation} \label{p-Cauchy-X0}
\begin{cases}
  \p_t u = \sum_{j=1}^m X_j^2 u & \quad (x,t) \in  S_T, \\
  u ( x, 0) = u_0(x)\geq 0 & \quad x \in \R^N,\\
  u (x,t)\geq 0 & \quad (x,t) \in S_T,
\end{cases}
\end{equation}
admits at most one solution.}
\end{theorem}
We note that the first uniqueness result for the positive Cauchy problem was established by Widder for the classical
heat equation in the Euclidean space \cite{Widder}.

The proof of Theorem~\ref{th-main-2} relies on Theorem~\ref{th-main} which, under the additional assumption $X_0=0$,
asserts that every nonnegative extremal solution $u$ of $\L u = 0$ in $\OT$ satisfies
\begin{equation*} %\label{eq_funct_eq-nodrift}
    u(x,t) = \mathrm{e}^{-\lambda t} u_0(x) \qquad \forall (x,t) \in \OT,
\end{equation*}
where $\lambda \leq \lambda_0$, and $\lambda_0$ is the generalized principal eigenvalue (see Remark~\ref{r-separation}).

\medskip

Before giving the proof of Theorem~\ref{th-main-2}, we should compare it with a result of Chiara Cinti
\cite{Cinti09} who considered a class of left translation invariant hypoelliptic operators with nontrivial drift
$X_0$ under the additional hypothesis that the operator is {\em homogeneous} with respect to a group of dilations on
the underlying Lie group. The method used in \cite{Cinti09} relies on some accurate upper and lower bounds of the
fundamental solution of $\L$. We note that the lower bounds for the fundamental solution are usually obtained by
constructing suitable \emph{Harnack chains}, as the ones used in the proof of Theorem~\ref{th-main}.
On the other hand, in order to apply the method used in \cite{Cinti09},  the upper and lower bounds need
to agree asymptotically. Hence, the Harnack chains need to be chosen in some optimal way. An advantage of our method is
that it does not require such an optimization step. Actually, a priori bounds of the fundamental solution, and even its
existence are not needed. We also note that the bibliography of \cite{Cinti09} contains an extensive discussion of
known results on the uniqueness of the Cauchy problem. We also recall a recent result by Bumsik Kim \cite{Kim} for the heat equation associated with subelliptic diffusion operators. In  his work, Kim proves uniqueness results for the heat equation under curvature bounds through the generalized curvature-dimension criterion developed by Baudoin and Garofalo and thus without the Lie group assumption.

\medskip

We start the proof of Theorem~\ref{th-main-2} with some preliminary results that do not require the assumption $X_0=0$.

Consider the positive Cauchy problem
\begin{equation} \label{p-Cauchy}
\begin{cases}
  \L u(x,t) = 0 & \quad (x,t) \in  S_T, \\
  u ( x, 0) = u_0(x) & \quad x \in \R^N,\\
  u (x,t)\geq 0 & \quad (x,t) \in S_T,
\end{cases}
\end{equation}
with $u_0\geq 0$ continuous function in $\R^N$.

We first recall some basic results on hypoelliptic operators of the form \eqref{e1}. Usually, hypoelliptic
operators have been studied under the further assumption that it is \emph{non-totally degenerate}, namely,
there exists a vector $\nu \in \R^N$ and  $j \in \big\{1, \dots, m\big\}$ such that
\begin{equation} \label{e-Bony}
  \langle X_j(x), \nu \rangle \ne 0\qquad  \mbox{ for all } x \in \R^N.
\end{equation}
This condition was introduced by Bony in \cite{Bony} and is not very restrictive. We also refer to
\cite{BonfiglioliLanconelli-2012} for a weaker version of this condition.

We observe that \eqref{e-Bony} can be always satisfied by a simple \emph{lifting} procedure. Indeed, let $\L$ be of the
form \eqref{e1}, and consider the operator $\widetilde \L$ acting on $(x_0, x, t) \in \R^{N+2}$ and defined by
\begin{eqnarray*}
     \widetilde \L u : = - \partial_{x_0}^2 u + \L u = \p_t u - \partial_{x_0}^2 u - \sum_{j=1}^m X_j^2 u + X_0 u.
\end{eqnarray*}
Clearly, $\widetilde \L$ is non-totally degenerate with respect to $\nu=(1,0,\ldots,0)\in\R^{N+1}$. Moreover, $\widetilde \L$ is
hypoelliptic and satisfies (H1) and (H2) if $\L$ is hypoelliptic and satisfies (H1) and (H2). Our uniqueness result for
$\L$ readily follows from the uniqueness for $\widetilde \L$. Therefore, in the sequel we assume that $\L$ satisfies \eqref{e-Bony}.

\medskip

We recall Bony's {\em strong maximum principle} \cite[Th\'eor\`{e}me 3.2]{Bony} for hypoelliptic operators
$\L$ of the form \eqref{e1} that satisfy  \eqref{e-Bony}. With our notation, it reads as follows. \emph{Let
$\Omega$ be any open subset of $\R^{N+1}$ and let $u \in C^{2}(\Omega)$ be such that $\L u
\ge 0$ in $\Omega$. Let $z_0 \in \Omega$ be such that $u(z_0) = \max_{\O}u$. If $\gamma: [0,T_0] \to \Omega$ is an
$\L$--admissible path such that $\g(0) = z_0$, then $u(\g(s)) = u(z_0)$ for every $s \in [0,T_0]$.}

The following {\em weak maximum principle} can be obtained as a consequence of Bony's strong maximum
principle. \emph{Let $\Omega$ be any bounded open set of $\R^{N+1}$ and let $u \in C^{2}(\Omega)$ be such that $\L u
\leq 0$ in $\Omega$. If $\limsup_{\underset{z\in \Omega}{z \to w}} u (z) \le 0$ for every $w \in \partial \Omega$, then
$u \le 0$ in $\Omega$.}

Let $\Omega$ be any bounded open set of $\R^{N+1}$, and let $\varphi \in C(\partial \O)$. The axiomatic potential
theory provides us with the Perron solution $u_\varphi$ of the boundary value problem $\L u = 0$ in $\O$, $u = \varphi$
in $\partial \O$. It is known that $u_\varphi$ might attain the prescribed boundary data only in a subset of $\partial
\O$. We say that $w \in \partial \O$ is \emph{regular} for $\L$ if $\lim_{\underset{z\in \Omega}{z \to w}} u_\varphi (z)
\to \varphi(w)$ for every $\varphi \in C(\partial \O)$. we denote by $\partial_r(\Omega)$ the set of the regular points
of $\partial \O$
\begin{equation*}
  \partial_r(\Omega) := \big\{ w \in \partial \O \mid \lim_{\underset{z\in \Omega}{z \to w}} u_\varphi (z) \to
\varphi(w)  \
\text{ for every }  \varphi \in C(\partial \O) \big\}.
\end{equation*}

Under assumption \eqref{e-Bony} it is possible to construct a family of \emph{regular cylinders} of $\R^{N+1}$, that is
cylinders such that their regular boundary agree with their \emph{parabolic boundary} \cite{LanconelliPascucci}.
Specifically, we denote by
$B(x,r)$ the Euclidean ball centered at $x \in \R^N$ with radius $r$. Let $\nu$ be a vector satisfying \eqref{e-Bony},
and assume, as it is not restrictive, that $|\nu|=1$. For every $x \in \R^N$ and $k \in \N$ we set
\begin{equation*}
  B_k(x) := B(x + k \nu, 2k) \cap B(x - k \nu, 2k).
\end{equation*}
It turns out that for every $x\in \R^N$, $k\in \N$, and $0<T_0<\infty$, the cylinder $Q_{k,T_0} (x) := B_k(x) \times
\,]0,T_0[$ is regular, see \cite{LanconelliPascucci} for a detailed proof of this statement.

We note that the sequence of regular cylinders  $\left( Q_{k,T_0} (0)\right)_{\underset{k \in \N}{0<T_0<T}}$ exhausts
the set $S_T$. This property will be used in the sequel.

\medskip

Consider a regular cylinder $Q := B \times\, ]0,T_0[$ and a function $f \in C(\overline Q)$. In \cite[Theorem
2.5]{LanconelliPascucci} it is proved that for a hypoelliptic operator $\L$ of the form \eqref{e1} satisfying
\eqref{e-Bony} there exists a unique solution $u \in C^{\infty}(Q) \cap C(Q\cup \partial_r(Q))$ to the following
initial-boundary value problem
\begin{equation} \label{p-Dirichlet}
\begin{cases}
\L u= f &  \quad \text{in} \  Q, \\[2mm]
u =  0 &  \quad \text{in} \ \partial_r Q.
\end{cases}
\end{equation}
We next show that the same result holds when a continuous compactly supported initial condition is prescribed on the
bottom of $Q$.

\begin{lemma} \label{lem-bvp}
Let $\L$ be a hypoelliptic operator of the form \eqref{e1} satisfying \eqref{e-Bony}, and let $Q := B \,\times\,
]0,T_0[$ be a regular cylinder. Let $\varphi \in C(B)$ be such that \emph{supp}$(\varphi) \subset B$. Then
there exists a unique $u \in C^{\infty}(Q) \cap C(Q\cup \partial_r Q)$ to the following initial-boundary value
problem
\begin{equation} \label{p-Dirichlet-c}
\begin{cases}
\L u= 0 &  \quad \text{in} \  Q, \\
u (x,t) =  0 &  \quad \text{in} \ \partial B \times [0,T_0], \\
u (x,0) = \varphi(x) &  \quad \text{in} \ B \times \big\{ 0 \big\}.
\end{cases}
\end{equation}
\end{lemma}

\begin{proof}
We use a standard argument. Consider, for any positive $\e$, a function $w_\e \in C^\infty \left(\overline{Q}\right)$
such that $w_\e(\cdot, 0) \to \varphi$, uniformly as $\e \to 0$, and takes the zero boundary condition at the lateral
boundary of $Q$. Denote by $f_\e := \L w_\e$, and note that $f_\e$ is continuous on $\overline Q$. We recall that we
can solve uniquely the initial-boundary value problem of the form \eqref{p-Dirichlet}. So, let $v_\e$ be the unique
solution of
the following problem
\begin{equation*}
\begin{cases}
\L v_\e= f_\e &  \quad \text{in} \  Q, \\
v_\e =  0 &  \quad \text{in} \ \partial_r Q.
\end{cases}
\end{equation*}
The function $u_\e : = w_\e - v_\e$ is clearly the unique solution of
\begin{equation*}
\begin{cases}
\L u_\e= 0 &  \quad \text{in} \  Q, \\
u_\e (x,t) =  0 &  \quad \text{in} \ \partial B \times [0,T_0], \\
u_\e (x,0) = w_\e (x,0) &  \quad \text{in} \ B \times \big\{ 0 \big\}.
\end{cases}
\end{equation*}
By the maximum principle, $u_\e$ uniformly converges to a continuous function $u$ that is a classical
solution of \eqref{p-Dirichlet-c}. The uniqueness follows from the weak maximum principle.
\end{proof}

Next, we apply the well-known argument (introduced by Donnelly for nondegenerate parabolic equations \cite{Donnelly}) to
show that the uniqueness for the positive Cauchy problem is equivalent to the uniqueness of the positive Cauchy problem
with the {\em zero} initial condition. For this sake, we prove the following proposition, which clearly implies the above
equivalence.

\begin{proposition}\label{prop_Donn} Let $\L$ be a hypoelliptic operator of the form \eqref{e1}, satisfying
\eqref{e-Bony}. If $u \in C(\overline {S_T}) \cap C^\infty(S_T)$ is a solution of the positive Cauchy problem
\eqref{p-Cauchy}, then there exists a minimal nonnegative solution $\widetilde u$ of \eqref{p-Cauchy}. Namely, $0 \le
\widetilde u \le v$ in $S_T$ for any solution $v$ of \eqref{p-Cauchy}.
\end{proposition}

\begin{proof}
 We use a standard exhaustion argument. Consider a sequence of continuous functions $\psi_k: \R^N \to \R$
such that $0 \le \psi_k(x) \le 1$, for any $k \in \N$, and that $\psi_k(x) = 1$ whenever $|x| \le k$, and $\psi_k(x) =
0$ if $|x| \ge k+1$. Consider a sequence $Q_k := \Omega_k \times\, ]0,T_k[$ of regular cylinders, such that supp$\left(
\psi_k \right) \subset \Omega_k$, and $T_k\nearrow T$. Let $\widetilde u_k$ be the solution to
\begin{equation*}
\begin{cases}
\L \widetilde u_k= 0 &  \quad \text{in} \  Q_k, \\
\widetilde u_k (x,t) =  0 &  \quad \text{in} \ \partial \Omega_k \times [0,T_k], \\
\widetilde u_k (x,0) = \psi_k(x)u_0(x) &  \quad \text{in} \ \Omega_k \times \big\{ 0 \big\},
\end{cases}
\end{equation*}
whose existence is given by Lemma \ref{lem-bvp}. By the comparison principle, $\left( \widetilde u_k \right)_{k
\in \N}$ is a nondecreasing sequence of nonnegative solutions of the equation $\L \widetilde u_k = 0$, such that
$\widetilde u_k (x,t) \le u(x,t)$.
Then, the function
\begin{equation*}
  \widetilde u (x,t) := \lim_{k \to \infty} \widetilde u_k (x,t)
\end{equation*}
is a distributional solution of $\L \widetilde u = 0$ in $S_T$ such that $0 \le \widetilde u \le u$ in $S_T$. By the
hypoellipticity of $\L$, $\widetilde u$ is a  smooth classical solution of the equation $\L \widetilde u = 0$ in $S_T$.
In order to prove that $\widetilde u$
takes the initial condition, we fix any $x_0 \in \R^N$, and we choose $k_0 > |x_0|$. We have
\begin{equation*}
  \widetilde u_{k_0} (x,t) \le \widetilde u (x,t) \le u (x,t),
\end{equation*}
for every $(x,t) \in Q_k$ with $k \in \N$. Consequently, $\widetilde u(x,t) \to u_0(x_0)$ as $(x,t) \to (x_0,0)$,
and this concludes the proof.
\end{proof}

\begin{corollary}\label{cor_uniquenesspositive} The positive Cauchy problem has a unique solution if and only if
any nonnegative solution of the positive Cauchy problem with $u_0 = 0$ is  the trivial solution $u=0$.
\end{corollary}

%\mysection{Proof of Theorem~\ref{th-main-2}}\label{sec_th-main-2}
%The present section is devoted to

In the following proof of Theorem~\ref{th-main-2}, which relies on Choquet's integral representation
theorem and the separation principle \eqref{eq_funct_eq-nodrift}, we resume the assumption $X_0=0$.

\begin{proof}[Proof of Theorem~\ref{th-main-2}] By Corollary \ref{cor_uniquenesspositive}, we may assume that $u_0 = 0$.

So, let $S_T = \R^N \times\, ]0,T[$ with $0 < T \le + \infty$, and let $u: S_T \to \R$ be a solution of the positive
Cauchy problem
\begin{equation} \label{e-Cauchy}
\begin{cases}
\L u (x,t)= 0 &  \quad (x,t) \in S_T, \\
u ( x, 0) = 0 & \quad x \in \R^N,\\
u (x,t)\geq 0 &  \quad (x,t) \in S_T.
\end{cases}
\end{equation}
We need to prove that $u = 0$.

As in \cite{KoranyiTaylor85}, we extend the solution $u$ of the Cauchy problem \eqref{e-Cauchy} to the whole domain
$\OT$ by setting
$$\tilde{u}(x,t) := \left\{
             \begin{array}{ll}
               u(x,t) & \quad t \in [0,T[, \\
               0 &  \quad t<0.
             \end{array}
           \right. $$
It is easy to see that $\tilde{u}$ is a distributional solution of $\L u = 0$ in $\OT$. Hence, the hypoellipticity of
$\L$ yields that $\tilde{u}$ is a nonnegative smooth classical solution of the equation \begin{equation}
\label{e-Cauchy-2}
    \L w= 0  \qquad \mbox{in }\OT,
\end{equation}
and $\tilde{u} = 0$ in $\R^N \times \R^-$. We need to prove that $\tilde{u}=0$ in $S_T$.

Suppose that $u\neq 0$, and let $a \in C( ]\!-\infty, T[)$ be a nonnegative function such that
$\tilde{u}\in \H_{a}^1$. By Choquet's integral representation theorem  and \eqref{e-exH}, it follows that  $\tilde{u}$
can be represented as
\begin{equation}\label{int_rep}
    \tilde{u}(x,t) = \int_{\H_+}v(x,t) d \mu (v)
\end{equation}
for some probability measure $\mu$ supported on $\big\{ 0 \big\} \cup \big\{ \exr \H_+ \cap \H_{a}^1 \big\}$. Recall
that $\tilde{u}(x,t)=0$ for $t\leq 0$. On the other hand, by \eqref{eq_int_curv}, any nonnegative solution $v\in \big\{
\exr \H_+ \cap \H_{a}^1 \big\}$ is strictly positive in a neighborhood of an integral curve of the form
$$\gamma:=\big\{ \exp\left( s \left(\omega \cdot X + Y \right) \right) z_0 \mid s \in \;] t_0- T, + \infty[\big\},$$
where $z_0 = (x_0,t_0)$ might depend on $v$. In particular, all such $v$ are strictly positive in $\R^N \times \R^-$.
Therefore, \eqref{int_rep} implies that
\begin{equation*}
    \mu \big\{ \exr \H_+ \cap \H_{a}^1 \big\} = 0.
\end{equation*}
Hence,  $\tilde{u} = 0$.
\end{proof}

\mysection{Mumford operator}\label{sec_mumford}
% \setcounter{equation}{0}
% \setcounter{theorem}{0}
%%%%%%%%%%%%%%%%%%%%%%%%%%%%%%%%%%%%%%%%%%%%%%%%%%%%%%%%%
The Mumford operator $\mathscr{M}$ is defined as
\begin{equation} \label{ex-Mum}
    \mathscr{M} u := \p_{t}  u - \cos(x) \p_y u - \sin(x) \p_w u - \p_x^2 u  \qquad (x,y,w,t) \in \R^4.
\end{equation}
It models the relative likelihood of different edges disappearing in some scene to be matched up by some hidden edges,
and explains the role of \emph{elastica} in computer vision \cite{Mumford}. In the present section we prove the
uniqueness of the positive Cauchy problem for $\M$, and we establish some properties of the minimal positive solutions
of $\M u = 0$. The following proposition allows us to apply our results to $\M$.

\begin{proposition} \label{prop-mum}
The Mumford operator $\M$ satisfies conditions {\rm(H0)} with the group operation
\begin{equation} \label{mumgl}
\begin{split}
 (x_0,y_0,w_0,t_0) \, \circ & \, (x,y,w,t) := \\
 & \big(x_0 + x, y_0 + y \cos (x_0)  - w \sin (x_0), \\
 & \qquad w_0 + y \sin (x_0) + w \cos (x_0), t_0 + t\big)
\end{split}
\end{equation}
for every $(x_0,y_0,w_0,t_0), (x,y,w,t) \in \R^4$. Moreover, $\M$ satisfies {\rm(H2)} with $\omega \ne 0$.
\end{proposition}

\begin{proof} Condition (H0) is verified by a direct computation. Moreover, it is known that $\M$ is invariant with
respect to the left translations of the group $\mathbb{G}:= (\R^3\times \R, \circ)$ on $\R^4$ whose operation is defined
by \eqref{mumgl}, (see \cite[Formula (61)]{BonfiglioliLanconelli-2012}).
$\mathbb{G}$ is called in the literature the \emph{roto-translation group}.

In order to check (H2), we note that
\begin{equation} \label{eq-gamma-0-Mum}
  \exp\left(s Y \right)  (x,y,w,t) = \left(x ,y + s \cos(x), w + s \sin(x), t-s\right),
\end{equation}
where $Y = \cos(x) \p_y  + \sin(x) \p_w -\p_{t}$  (see \eqref{eY}), while
\begin{equation} \label{eq-gamma-Mum}
\begin{split}
  & \exp\left(s (\omega X + Y)\right)  (x,y,w,t) = \\
  & \qquad \qquad \left(x + s \omega,y +
\frac{\sin(x + s \omega) - \sin(x)}{\omega},w - \frac{\cos(x + s \omega) - \cos(x)}{\omega},t-s\right),
\end{split}
\end{equation}
for every $(x,y,w,t) \in \R^4$, and $s,\omega \in \R$, with $\omega \ne 0$.

We first show that
\begin{equation} \label{eq-prop-Mum}
  \A_{z_0} = \big\{ (x,y,w,t) \in \R^4 \mid \sqrt{(y-y_0)^2 + (w-w_0)^2}\le t_0-t \big\},
\end{equation}
for every $z_0 = (x_0,y_0,w_0,t_0) \in \R^4$. The inclusion $\A_{z_0}$ in the right hand side of \eqref{eq-prop-Mum}
follows directly from the definition of attainable set, and from the fact that the norm of the drift term $X_0 = \cos(x)
\p_y + \sin(x) \p_w \simeq (0, \cos(x), \sin(x), 0)$ equals $1$.

We next prove the inclusion of the right hand side of \eqref{eq-prop-Mum} in $\A_{z_0}$. We first note that, by the
invariance with respect to the Lie operation \eqref{mumgl}, it is not restrictive to assume that $(x,y,w,t) = 0$. We
also assume that $(y_0, w_0) \ne (0,0)$ since $\A_{z_0}$ is the closure of the set of the reachable points. We
introduce
polar coordinates; $\widetilde x = - \arg(y_0,w_0)$, and $\widetilde t = \sqrt{y_0^2 + w_0^2}$, and we note that
\begin{equation} \label{eq-tilde}
  (y_0, w_0) = - \widetilde t (\cos(\widetilde x), \sin(\widetilde x)) \qquad 0 < \widetilde t \le t_0.
\end{equation}
We define the sequence of paths $\left( \g_k \right)_{k \in \N}$ in the interval $[0, \widetilde t]$ by choosing
\begin{equation*} %\label{eq-prop-Mum}
  x_k(0) = x_0, \qquad x_k(\widetilde t) = 0, \qquad x_k(s) = \widetilde x, \quad \text{for} \quad \frac{\widetilde t}{4
k} \le s \le
\left( 1 - \frac{1}{4 k}\right)\widetilde t,
\end{equation*}
and $x_k$ linear in $\left[ 0, \frac{\widetilde t}{4 k} \right]$ and in $\left[ \left( 1 - \frac{1}{4
k}\right)\widetilde t, \widetilde t \right]$.
If $\widetilde t < t_0$, we set $x_k(s) = 2 \pi (s - t_0 + \widetilde t)/\widetilde t$, for every $s \in [t_0 -
\widetilde t, t_0]$. Moreover,
\begin{equation*} %\label{eq-ywt-tilde}
  y_k(s) = y_0 + \int_0^s \cos(x_k(\tau)) d \tau, \quad w_k(s) = w_0 + \int_0^s \sin(x_k(\tau)) d \tau, \quad t_k(s)
= t_0 -s.
\end{equation*}
We clearly have that $x_k(t_0) = 0, t_k(t_0)=0$. Moreover, a simple computation based on \eqref{eq-tilde} gives
$|y_k(t_0)| = |y_k(\widetilde t)| \le \frac{1}{2k}(|y_0| + \widetilde t) \le \frac{1}{2k}(|y_0| + t_0)$ and,
analogously, $|w_k(t_0)| \le \frac{1}{2k}(|w_0| + t_0)$.
This proves that $\g_k(t_0) \to 0$ as $k \to + \infty$. In particular $0 \in \A_{z_0}$, and the proof of
\eqref{eq-prop-Mum} is completed.

The above argument also applies to any bounded open box $\Omega$ which is sufficiently wide in the $x$-direction.
More precisely, if $\Omega = ]x_0-R_x,x_0+R_x[ \times ]y_0-R_y,y_0+R_y[ \times ]w_0-R_w,w_0+R_w[ \times
]t_0-R_t,t_0+R_t[$ with $R_x > \pi$, then
\begin{equation*} %\label{eq-prop-Mum-H2}
  \A_{z_0}(\Omega) = \big\{ (x,y,w,t) \in \Omega \mid \sqrt{(y-y_0)^2 + (w-w_0)^2}\le t_0-t \big\}.
\end{equation*}

Note that, by \eqref{eq-gamma-0-Mum} and \eqref{eq-gamma-Mum}, we have that $\exp\left(s (\omega X + Y)\right)
(z_0)$ belongs to the interior of $\A (z_0)(\Omega)$ if, and only if, $\omega \ne 0$. This proves (H2).
\end{proof}
We next prove a separation principle for the extremal solutions of the equation $\M u = 0$. We have
\begin{proposition} \label{prop-sep-mum}
For every $u \in \exr \H_+$ there exist two constants $\beta \in \R$ and $C_0>0$ such that
\begin{equation*}
 u(x + 2 k \pi,y,w,t) = C_0^k e^{\beta t} u(x,y,w,0) \qquad \text{for every} \quad (x,y,w,t) \in \OT, \ k \in \Z.
\end{equation*}
In particular for $k=0$, we have
\begin{equation*}
 u(x,y,w,t) = e^{\beta t} u(x,y,w,0) \qquad \text{ for every} \quad (x,y,w,t) \in \OT.
\end{equation*}

\end{proposition}

\begin{proof} We first prove that
\begin{equation} \label{eq-sep-mum}
 u(x,y,w,t - s) = e^{-\beta s} u(x,y,w,t) \quad \text{for every} \quad (x,y,w,t) \in \OT, \ s > 0.
\end{equation}
Fix any positive $s$, and choose $\omega = 2 \pi / s$. Recall \eqref{eq-gamma-Mum}, and note that
\begin{equation*}
  \exp \left(s (- \omega X + Y) \right) \left( \exp\left(s (\omega X + Y)\right) (x,y,w,t) \right) = (x , y, w,t-2 s),
\end{equation*}
and that the change of variable  $(x,y,w,t) \mapsto (x , y, w,t-2 s)$ preserves the equation $\M u = 0$. Then the
hypotheses of Proposition \ref{prop-separation} are satisfied with $\omega_1 = - \omega_2 := \omega$ and $s_1 = s_2 :=
s$. Hence we have
\begin{equation*}
  u (x , y, w,t-2 s) = C  u (x , y, w,t)
\end{equation*}
for some positive constant $C=C(s)$. Hence, \eqref{eq-sep-mum} followed as in the last part
of the proof of Theorem~\ref{th-main}.

In order to conclude the proof, we consider again a positive $s$, we set $\omega = 2 \pi / s$, and we note that
\begin{equation*}
  \exp\left(s (\omega X + Y)\right) (x,y,w,t) = (x + 2 \pi, y, w,t-s).
\end{equation*}
Also in this case the assumptions of Proposition \ref{prop-separation} are satisfied with $\omega_1 := \omega$ and
$s_1 := s$, thus there exists a positive constant $C$ such that
\begin{equation*}
  u (x + 2 \pi, y, w,t - s) = C  u (x , y, w,t) \quad \text{for every} \quad (x,y,w,t) \in \OT.
\end{equation*}
The conclusion of the proof then follows by combining the above identity  with \eqref{eq-sep-mum}.
\end{proof}

The following result is a corollary of Proposition \ref{prop-sep-mum}.

\begin{theorem} \label{th-mum-PCP}
Let $\M$ be the Mumford operator \eqref{ex-Mum}, and let $u_0 \geq 0$ be a continuous function in $\R^3$. Then the
positive Cauchy problem
\begin{equation*}
\begin{cases}
  \M u (x,y,w,t) = 0 & \quad (x,y,w,t) \in  S_T, \\
  u (x,y,w,0) = u_0(x,y,w) & \quad (x,y,w) \in \R^3,\\
  u (x,y,w,t) \geq 0 & \quad (x,t) \in S_T,
\end{cases}
\end{equation*}
admits at most one solution.
\end{theorem}

\begin{proof}
 The proof is exactly as in the proof of Theorem \ref{th-main-2}, once the separation principle
\eqref{eq-sep-mum} has been established. We omit the details.
\end{proof}
%%%%%%%%%%%%%%%%%%%%%%%%%%%%%%%%%%%%%%%%%%%%
\mysection{Kolmogorov-Fokker-Planck operators} \label{sec_Kolmogorov}
Consider the Kolmogorov operator
\begin{equation}\label{eq-Kolmogorov-md}
  \L u (x,y,t) := \partial_t u(x,y,t) - \sum_{j=1}^m \partial^2_{x_j} u(x,y,t) - \sum_{j=1}^m x_j \partial_{y_{j}}
u(x,y,t),
\end{equation}
with $(x,y,t) \in \R^m \times \R^m \times \R$. As usual, we denote $\OT=\R^{2m} \times ]-\infty, T[$. The operator
$\L$ can be written in the form \eqref{e1} by setting $X_j: = \partial_{x_j}$ for $j= 1, \dots, m$, and $X_0 :=
\sum_{j=1}^m x_j \partial_{y_{j}}$. It follows that $\L$ satisfies H\"ormander's condition (H0). The vector fields
$X_j$'s and $Y := X_0-\partial_t$ are invariant with respect to the left translations and the dilation defined by
\begin{equation} \label{kolmogl}
    (\xi, \eta, \tau) \circ (x,y,t) := (x + \xi, y + \eta - t \xi ,t + \tau),
    \qquad \delta_r (x,y,t) := (r x, r^3 y, r^2 t),
\end{equation}
respectively. An invariant Harnack inequality for Kolmogorov equations was first proved by Garofalo and Lanconelli in
\cite{GarofaloLanconelli}. It can be written in its restricted form as in Proposition~\ref{p-restr-harnack} with
$\omega = 0$. It reads as
\begin{equation} \label{e-harnack-kolmogorov}
  u\left( x, y+\t x, t-\t \right) \le C_{\t} \, u(x,y,t) \qquad \text{for every} \ (x,y,t) \in \R^{2m+1}  \text{ and }
\t >0.
\end{equation}
We stress that due to the drift term $X_0 - \partial_t$, the Harnack inequality for Kolmogorov equations is
different from \eqref{e-harnack-Heisenberg}. The above discussion applies to the following more general class of
operators of the above type, first studied by Lanconelli and Polidoro in \cite{LanconelliPolidoro94}. We also refer to
the book by Lorenzi and Bertoldi \cite{LorenziBertoldi} and to the bibliography therein for results on Kolmogorov
equations obtained by semigroup theory.

We summarize the properties of $\mathscr{L}$ that are needed for its study in our functional setting. Condition
(H0) can be verified by a direct computation, while the group operation required to satisfy (H1) is defined
in  \eqref{kolmogl}. Condition (H2) holds for every $\omega \in \R^m$. In the
sequel we choose $\omega = 0$.

We use the explicit expression of the fundamental solution $\Gamma$ of $\L$ to compute the Martin functions of $\R^{2m} \times
]-\infty, T[$. We recall that this method has been used in \cite{CranstonOreyRosler} (see (1.2) therein)
to compute the complete parabolic and elliptic Martin boundary for nondegenerate Ornstein-Uhlenbeck processes in
dimension two (see also \cite{Doob} for other explicit examples of computing parabolic Martin boundaries).

We recall the definition of Martin functions for our case.  Assume for simplicity that $T<\infty$. We say that a
sequence $\{(\xi_k, \eta_k, \tau_k)\}_{k\in \N}$ is a
\emph{fundamental sequence} if $\| ( \xi_k, \eta_k, \tau_k )\| \to + \infty$ as $k \to \infty$ and the corresponding
sequence of {\em Martin quotients} $\{u_k\}$ given by
\begin{equation}\label{eq-MBlim}
  u_k(x,y,t)  := \frac{\Gamma(x,y,t, \xi_k, \eta_k, \tau_k)}{\Gamma(0,0,T, \xi_k, \eta_k, \tau_k)}
 \end{equation}
converges to a nonnegative solution $u(x,y,t) := \lim_{k \to \infty} u_k(x,y,t)$ in $\H_+$. Such a $u$ is called a {\em
Martin function}
$u$ of $\L$ in $\OT$. It is a is a nonnegative solution of $\L u = 0$ in $\OT$ which is defined by
some fundamental sequence $(\xi_k, \eta_k, \tau_k)_{k\in \N}$. Note that $\Gamma(0,0,T, \xi_k, \eta_k, \tau_k) = 0$
whenever $T\leq \tau_k$, hence we need to assume $T> \tau_k $ for every $k \in \N$.

The explicit form of the fundamental solution $\Gamma$ of Kolmogorov operator is known and is given by
\begin{equation} \label{e-KolmogorovFS}
\Gamma(x,y,t, \xi, \eta, \tau) = \left( \tfrac{3}{2 \pi}\right)^{m/2} \!\! \dfrac{1}{(t- \tau)^{2m}}  \exp\! \left(
\!\! - \frac{\|x- \xi\|^2}{4(t- \tau)} - 3 \frac{\|y - \eta + \tfrac{t- \tau}{2} (x +\xi)\|^2}{(t- \tau)^3} \right)
\end{equation}
if $t > \tau$, while $\Gamma(x, y, t, \xi, \eta, \tau) = 0$ if $t \le \tau$.

\medskip

 We have
\begin{proposition}\label{p_Kolmo}
Let $\L$ be the Kolmogorov operator \eqref{eq-Kolmogorov-md}, and let $u$ be a Martin function for $\L u =
0$ in $\OT$. Then either $u = 0$, or there exists $v \in \R^m$ such that
\begin{equation} \label{eq-exp-Kolmo}
  u(x,y,t) = \exp \left( \langle x, v \rangle + t \|v\|^2 \right) \quad \text{for all} \ (x,y,t) \in \OT.
\end{equation}
\end{proposition}
Since in any Bauer harmonic space all the extremal solutions are Martin kernels (see, Proposition 4.1 and Theorem 5.1 in
\cite{Maeda91}), we have

\begin{corollary} \label{c_Kolmo} Any nonnegative solution $u=u(x,y,t)$ of the Kolmogorov equation $\L u = 0$ in $\OT$
does not depend on the variable $y$, and $u$ is a nonnegative solution of the heat equation $\partial_t w(x,t) =
\varDelta w(x,t)$ in $\R^m\times \R_T$.

In particular, any nonzero nonnegative solution of the equation $\L u = 0$ in $\OT$ is strictly positive, and the
uniqueness of the positive Cauchy problem in $S_T$ holds true.

\end{corollary}
The uniqueness of the positive Cauchy problem in $S_T$ for the Kolmogorov equation was first
proved in \cite{Polidoro1995} by a different method.

\begin{proof}[Proof of Proposition~\ref{p_Kolmo}] Assume, as it is not restrictive, that $T=0$, let $u$ be a Martin
functions of $\L$ in $\OT$, and let $(x,y,t) \in \OT$. In order to prove our claim, we preliminarily note that
\begin{equation} \label{eq-firstterm-Martin}
  - \frac{1}{4}\left( \frac{\|x-\xi_k\|^2}{t - \tau_k} - \frac{\|\xi_k\|^2}{- \tau_k}  \right) =
  - \frac{1}{4}\left( \frac{\|x \|^2}{t - \tau_k} - 2 \frac{\langle x, \xi_k \rangle }{t - \tau_k}
- t \frac{\|\xi_k\|^2}{(t - \tau_k)(- \tau_k) } \right),
\end{equation}
and that
\begin{multline} \label{eq-secondterm-Martin}
   - 3 \left( \frac{ \| y - \eta_k + \tfrac{t- \tau_k}{2} (x +\xi_k)\|^2 }{(t - \tau_k)^3} - \frac{ \| \eta_k +
\tfrac{\tau_k}{2} \xi_k\|^2 }{(- \tau_k)^3} \right) = \\
  \qquad - 3 \frac{ \| y  + \tfrac{t}{2} x\|^2 }{(t - \tau_k)^3}
 - \frac{3}{4} \frac{\| t \xi_k - \tau_k x\|^2}{(t - \tau_k)^3}
 + 6 \frac{\langle y  + \tfrac{t}{2} x, \eta_k + \tfrac{\tau_k}{2} \xi_k \rangle}{(t - \tau_k)^3} \\
  \qquad - 3 \frac{\langle y  + \tfrac{t}{2} x, t \xi_k - \tau_k x \rangle}{(t - \tau_k)^3}
 + 3 \frac{\langle \eta_k + \tfrac{\tau_k}{2} \xi_k, t \xi_k - \tau_k x \rangle}{(t - \tau_k)^3} \\
  \qquad + \left( 9 t \tau_k^2 - 9 t^2 \tau_k + 3 t^3 \right) \frac{ \| \eta_k + \tfrac{\tau_k}{2}
\xi_k\|^2 }
 {(t -\tau_k)^3(- \tau_k)^3}\,.
\end{multline}

We next choose a fundamental sequence $\big((\xi_k, \eta_k, \tau_k)\big)_{k \in \N}$ such that $u(x,y,t) = 0$ for every
$(x,y,t) \in \OT$. We fix any vector $w \in \R^m$ such that $w \ne 0$, and we set $(\xi_k, \eta_k, \tau_k) = (k w, 0,
-1)$. Since $\Gamma(x, y, t, \xi, \eta, \tau) = 0$ if $t \le \tau$, we have $u_k(x,y,t)=0$ whenever $t < -1$. A direct
computation based on \eqref{eq-firstterm-Martin} and \eqref{eq-secondterm-Martin} shows that $u_k(x,y,t) \to 0$ also
if $-1 < t < 0$. We then conclude that $u = 0$ in $\OT$.

Note that, we find the trivial solution whenever a bounded subsequence of $\big(\tau_{k}\big)_{k\in \N}$ exists.
Indeed, let $\big(\tau_{k_j}\big)_{j\in \N}$ be a convergent subsequence of $\big(\tau_{k}\big)_{k\in \N}$, and denote
by  $\widetilde \tau \in ]- \infty, T]$ its limit. Let $(x,y,t) \in \R^{2m+1}$ be fixed, with $t <  \widetilde \tau$.
Then there exists a $J \in \N$ such that $\tau_{k_j} > t$, so that $u_{k_j}(x,y,t)= 0$ for every $j > J$. Thus
$u(x,y,t) = 0$ for every $(x,y,t)$ such that $t < \tilde \tau$. This proves the  claim if $\tilde \tau = T$. If
$\tilde \tau > T$ the uniqueness of the positive Cauchy problem for Kolmogorov equations (see Theorem 3.2 in
\cite{Polidoro1995}) implies that $u(x,y,t) = 0$ also when $\tilde \tau < t <T$. For this reason, in the sequel we
will always assume that $\tau_k \to - \infty$ as $k \to + \infty$.

We next show that \emph{nontrivial} Martin functions of $\L$ have the form \eqref{eq-exp-Kolmo}. We fix $w_1, w_2 \in
\R^m$ and we set $(\xi_k, \eta_k, \tau_k) = (2 k w_1, k^2 w_2, -k)$. A direct computation based on
\eqref{eq-firstterm-Martin} shows that
\begin{equation} \label{eq-martin-xk}
  - \frac{1}{4}\left( \frac{\|x-\xi_k\|^2}{t - \tau_k} - \frac{\|\xi_k\|^2}{- \tau_k}  \right) \to
  \langle x, w_1 \rangle + t \|w_1\|^2  \qquad \text{as} \quad k \to \infty.
\end{equation}
A similar argument, based on \eqref{eq-secondterm-Martin}, applies to last term in the exponent of
\eqref{e-KolmogorovFS}. We have
\begin{equation*}
    \eta_k + \tfrac{\tau_k}{2} \xi_k = k^2 \left( w_2- w_1 \right), \qquad
    t \xi_k- \tau_k x =  k \left(2 t w_1 - x \right),
\end{equation*}
then
\begin{equation*}
    y - \eta_k + \tfrac{t- \tau_k}{2} (x +\xi_k) =
    - k^2 \left( w_2- w_1 \right) + k  \left( t w_1 - \tfrac12 x\right) + y +\tfrac{t}{2} x.
\end{equation*}
Consequently, we find that
\begin{equation*}
\begin{split}
  & - 3 \left( \frac{\|y - \eta_k + \tfrac{t- \tau_k}{2} (x +\xi_k)\|^2}{(t - \tau_k)^3} -
  \frac{\| \eta_k + \tfrac{\tau_k}{2} \xi_k\|^2}{(- \tau_k)^3}  \right) = \\
  & - 3 \frac{-k^6 \langle w_2 - w_1 , 2 t w_1 + x \rangle - 3 t k^6 \|w_1 - w_2\|^2}{k^3(t + k)^3} + \omega(k),
\end{split}
\end{equation*}
for some function $\omega$ such that $\omega(k) \to 0$ as $k \to \infty$. Hence,
\begin{equation} \label{eq-martin-yk}
  - 3 \left( \frac{\|y - \eta_k + \tfrac{t- \tau_k}{2} (x +\xi_k)\|^2}{(t - \tau_k)^3} -
  \frac{\| \eta_k + \tfrac{\tau_k}{2} \xi_k\|^2}{(- \tau_k)^3}  \right) \to
  3 \langle w_2 - w_1 , 2 t w_1 + x \rangle + 9 t \|w_1 - w_2\|^2,
\end{equation}
as $k \to \infty$. Note that the variable $y$ doesn't appear in last limit. Thus, also using the
obvious fact $\left(\frac{-\tau_k}{t - \tau_k}\right)^{2m} \to 1$ as $k \to \infty$, we find
\begin{equation*}
  u(x,y,t) = \exp \left( \langle x, 3 w_2 -2 w_1 \rangle + t \|3 w_2 -2 w_1\|^2 \right),
\end{equation*}
and we conclude that $u$ has the form \eqref{eq-exp-Kolmo} if we choose $v = 3 w_2 -2 w_1$.

We next show that either $u$ is zero, or has the form \eqref{eq-exp-Kolmo}, for every fundamental sequence. With this
aim, we consider any sequence $(\xi_k, \eta_k, \tau_k)_{k\in \N}$, with $\tau_k <0$ for every $k \in \N$, and
such that $\tau_k \to - \infty$ as $k \to  + \infty$, since we know that, otherwise, $u$ is the trivial solution.
We also assume that the function $u$ in \eqref{eq-MBlim} is well defined.

We set
\begin{equation} \label{eq-ratio}
 \widetilde \xi_k := \tfrac{1}{-\tau_k} \xi_k , \qquad \widetilde \eta_k :=  \tfrac{1}{(-\tau_k)^2} \eta_k,
 \quad k \in \N.
\end{equation}
and, after some elementary, but lengthly computations, we find that
\begin{equation} \label{eq-xi-eta-k}
  u_{k}(x,y,t)=\exp\left( \left(\langle x,3 \tilde{\eta}_{k}-\tilde{\xi}_{k}\rangle
+t\| 3 \tilde{\eta}_{k}-\tilde{\xi}_{k}\|^{2} \right)(1 +R_{k})\right),
\end{equation}
where $R_{k}\to 0$ denotes a vanishing sequence. Thus, either $\big\| 3 \tilde{\eta}_{k}-\tilde{\xi}_{k} \big\|
\to + \infty$ as $k \to + \infty$, or the sequence $\big( 3 \tilde{\eta}_{k}-\tilde{\xi}_{k} \big)_{k \in \N}$ has a
bounded subsequence.

In the first case we plainly find $u(x,y,t) = 0$ for every $(x,y,t) \in \R^{m+1}$ with $t<0$.

In the second case there exists a subsequence $\big( 3 \tilde{\eta}_{k_j}-\tilde{\xi}_{k_j} \big)_{j \in \N}$
converging to some point $w \in \R^{m}$. From \eqref{eq-xi-eta-k} we have that
\begin{equation*}
  u(x,y,t) = \exp \left( \langle x, w \rangle + t \| w \|^2 \right),
\end{equation*}
and hence,  $u$ has the form \eqref{eq-exp-Kolmo}. This concludes the proof.
\end{proof}

\mysection{Concluding remarks and further developments}\label{sec_further}
As was stressed in Remark~\ref{r-separation}, our separation principle (Theorem~\ref{th-main}) gives a
valuable information concerning nonnegative solutions for operators $\L$ of the form
\begin{equation*}
  \L u = \partial_t u - \sum_{j=1}^m X_j^2 u,
\end{equation*}
and for Mumford's operator $\M$
\begin{equation*}
    \mathscr{M} u := \p_{t}  u - \cos(x) \p_y u - \sin(x) \p_w u - \p_x^2 u.
\end{equation*}
On the other hand, in recent years, operators of the form \eqref{e1} with $X_0\neq 0$ that satisfy (H0), (H1) and
(H2) have received considerable attention. It would be interesting to study their positivity properties using our functional analytic approach. We give here two
examples of such operators.

\begin{example}\label{ex3} {\sc Linked operators.}
Let $(\p_x + y \p_s)^2 + (\p_y - x \p_s)^2$ be the sub-Laplacian on the Heisenberg group given by \eqref{e-HG}, and let
$ x \p_w -\p_{t}$ be the first order term of the simplest Kolmogorov operator
\eqref{eq-Kolmogorov-md}, that is
\begin{equation*}
\L := \p_{t} - x \p_w - \p_x ^2 \qquad (x,w,t) \in \R^3.
\end{equation*}
Define
\begin{equation}\label{exlink}
\L := \p_{t} - x \p_w - (\p_x + y \p_s)^2 - (\p_y - x \p_s)^2  \qquad (x,y,s,w, t) \in \R^5.
\end{equation}
Note that the operator $\L$ acts on the variables $(x,y,s,t)$ as the heat equation on the Heisenberg group, and on the
variables $(x,y,w,t)$ as a Kolmogorov operator in $\R^3 \times \R$. It is easy to see that $\L$ satisfies the
H\"ormander condition. Moreover, it can be shown that there exists a homogeneous Lie group on $\R^5$ that \emph{links}
the Heisenberg group on $\R^4$ and the Kolmogorov group in $\R^3$, and such that $\L$ is invariant with respect to
this new Lie group.

The notion of a {\it link of homogeneous groups} has been introduced by Kogoj and Lanconelli in \cite{KogojLanconelli2,
KogojLanconelli4}. It gives a general procedure for the construction of sequences of homogeneous groups of arbitrarily
large dimension and step.
\end{example}

\begin{example}\label{ex4} Consider the following operator studied by Cinti, Menozzi and Polidoro
\cite{CintiMenozziPolidoro}
\begin{equation} \label{ex-CMP}
    \L u = \p_{t}  u - x \p_w u - x^2 \p_y u - \p_x^2 u   \qquad (x,y,w,t) \in \R^4.
\end{equation}
It is invariant with respect to the following Lie group operations
\begin{equation} \label{eq-group-CMP}
 (x,y,w,t) \circ (\x,\y,\omega,\t): = (x + \x, y +\eta +2 x \omega - \t x^2, w +\omega-\t x, t+\t),
\end{equation}
and verifies H\"ormander hypoellipticity condition, so, (H0) and (H1) are satisfied. Note that, in this case, the
drift term $X_0 := x^2 \p_y + x \p_w$ is essential for the validity of (H0). $\L$ is also invariant with respect to the
following dilation
\begin{equation} \label{eq-dil-CMP}
 \d_r (x,y,w,t): = \big(r x, r^{4}  y, r^3 w, r^2 t \big).
\end{equation}
We next show that the attainable set of the point $z_0 =(x_0, y_0, w_0, t_0)$ in $\R^4$ is
\begin{equation} \label{eq-propset}
  \A_{z_0} = \big\{ (x,y,w,t) \in \R^4 \mid t\le t_0, y_0\le y, (w-w_0)^2\le (y - y_0 ) (t_0-t) \big\}.
\end{equation}
To prove \eqref{eq-propset}, we recall that in \cite[Lemma 5.11]{CintiMenozziPolidoro} it has been shown that, if $z_0
= 0 \in \R^4$, and $\O = \,\big( ]-1,1[ \big)^4$ is the open unit cube in $\R^4$, then
\begin{equation*}
  \A_{0} (\Omega)  = \big\{ (x,y,w,t) \in \Omega \mid 0 \le y \le -t, w^2 \le - t y \big\}.
\end{equation*}
In accordance with \eqref{eq-dil-CMP}, we consider the $r$ dilation of $\Omega$
$$\d_r \Omega = \, \, ]\!-r, r[ \, \, \times \, \,  ]\!-r^4, r^4[  \,\,
\times  \,\,  ]\!-r^3, r^3[  \,\,  \times \, \,  ]\!-r^2, r^2[ \,.$$
By the dilation invariance of $\L$, we then have
\begin{equation*}
  \A_{0} = \bigcup_{r > 0} \A_{0} (\d_r \Omega) = \bigcup_{r > 0} \big\{ (x,y,w,t) \in \d_r \Omega
\mid  0 \le y \le - r^2 t, w^2 \le - t y \big\},
\end{equation*}
and we get \eqref{eq-propset} for $z_0 = 0$. Eventually, \eqref{eq-propset} for any $z_0 \in \R^4$ follows
from the invariance of $\L$ with respect to the translations defined in \eqref{eq-group-CMP}.

Note that the point $\exp\left( s Y \right) z_0 \not \in$ Int$(\A_{z_0})$, where $Y=x^2 \p_y + x \p_w-\p_t$ is defined
by \eqref{eY}. Since $\A_{z_0}(\Omega) \subset \A_{z_0}$, for every bounded set $\Omega \subset \R^4$, we
conclude that  (H2) is not satisfied if we choose $\omega = 0$. Nevertheless, $\L$ defined in \eqref{ex-CMP}
satisfies assumption (H2), for any $\omega \ne 0$ provided that we choose $\Omega$ big enough.
\end{example}

We note that the operator $\L$ in \eqref{ex-CMP} is an approximation of the Mumford operator \eqref{ex-Mum}.
Indeed, the Taylor expansion at $x=0$ of the drift term $X_0 = \cos(x) \p_y + \sin(x) \p_w$, leads us to approximate
$\mathscr{M}$ with
\begin{equation*}
    \widetilde  {\mathscr{M}} = \p_{t} - \left( 1 - \tfrac{x^2}{2}\right) \p_y - x \p_w - \p_x^2 .
\end{equation*}
Moreover, it can be easily checked that $u$ is a solution of the equation $\L u = 0$ (where $\L$ is the
operator defined by \eqref{ex-CMP}) if and only if the function $v(x,y,w,t) :=  u\left(x,-\frac{y}{2}-
t,w,t\right)$ is a solution of the equation $\widetilde  {\mathscr{M}} v = 0$, and the claim is verified.

\subsection{On the separation principle} \label{s-remark}

We discuss here the main assumption \eqref{eq_right_inv} of Theorem \ref{th-main}. We recall that it is satisfied
whenever $X_0 = 0$, and therefore, it is natural to study operators with $X_0\neq 0$ and a non-abelian $\mathbb{G}$
that still satisfy \eqref{eq_right_inv}. In order to discuss this question, we focus on the consequence of
\eqref{eq_right_inv}, that is
\begin{equation} \label{eq_funct_eq-reminder}
    u\left(\exp(s (\omega \cdot X +Y))(x,t)\right) = \mathrm{e}^{- \beta s} u(x,t) \quad \forall (x,t) \in \OT \mbox{ and }  \forall s > 0,
\end{equation}
where $u$ is a nonnegative extremal solution. The following result answers this question

\begin{proposition} \label{th-comm}
Let $\L$ be an operator of the form \eqref{e1}, satisfying {\rm (H0)}, {\rm (H1)} and {\rm (H2)}. Let $u: \OT \to \R$ be
a nonnegative smooth function satisfying \eqref{eq_funct_eq-reminder}. Then
\begin{equation} \label{eq_comm-jk0}
    [X_j, X_k ] u(x,t) = 0, \qquad \forall j,k = 0, 1, \dots, m, \mbox{ and } \forall (x,t)\in \OT.
\end{equation}
The same result holds for all higher-order commutators.

Moreover, if any nonnegative extremal solution in $\H_+$ satisfies \eqref{eq_funct_eq-reminder}, then the conclusion
\eqref{eq_comm-jk0} holds for any $u\in\H_+$.
\end{proposition}

As an application, we apply the above result to the degenerate Kolmogorov equations in two space variables $\K := \partial_t - x \partial_y
- \partial_x^2$, and let $\H_+$ the corresponding cone of nonnegative solutions in $\R^2 \times ]- \infty, T[$ . In this case $X_1 = \partial_x, X_0 = x \partial_y$, and Proposition \ref{th-comm} says that, if $u$ is
a nontrivial nonnegative extremal solution in $\H_+$ that satisfies
\eqref{eq_funct_eq-reminder}, then
\begin{equation*}
  [X_1, X_0]u(x,y,t) = [\partial_x, x \partial_y]u(x,y,t) = \partial_y u(x,y,t) =0.
\end{equation*}
Hence, $u$ does not depend on $y$. Therefore, $u$ is a nontrivial nonnegative solution of the heat equation $\partial_t
u = \partial_x^2 u$ in $\R\times ]-\infty,T[$, and in particular $u$ is strictly positive.

% An analogous statement holds when considering Kohn-Laplace operator on the Heisenberg group $\Delta_\mathbb{H} =
% \left(\partial_x - \tfrac{y}2 \partial_z\right)^2 + \left(\partial_y + \tfrac{x}2 \partial_z\right)^2$, with $(x,y,z)
% \in \R^3$. In this case $\L = \Delta_\mathbb{H} - \partial_t$,
% \begin{equation*}
%   \left[\partial_x - \tfrac{y}2 \partial_z, \partial_y + \tfrac{x}2 \partial_z \right] = \partial_z.
% \end{equation*}
% and we find again that any nontrivial nonnegative extremal solution $u$, satisfying \eqref{eq_funct_eq}, is strictly
% positive and does not depend on $z$. Also in this case  $u$ is a positive solution of the heat equation $\partial_x^2 u
% + \partial_y^2 u = \partial_t u$ in $\R^2\times ]-\infty,T[$.

In conclusion, all nontrivial nonnegative extremal solutions in $\H_+$ satisfying \eqref{eq_funct_eq-reminder}, do not
depend on the 'degenerate' variable $y$. Recall that in fact, by Corollary~\ref{c_Kolmo},  all solutions in $\H_+$ do
not depend on $y$.

\medskip

Next, we present the proof of Proposition~\ref{th-comm}. It relies on the following Lemma, whose proof is analogous to that
of Theorem \ref{H*-lambda-repr-par}.

\begin{lemma} \label{lem-comm} Let $u: \OT \to \R$ be a nonnegative smooth function, and $\omega_1, \omega_2$ be two
vectors of $\R^m$ such that \eqref{eq_funct_eq-reminder} holds. Then, for every $(x,t) \in \OT$, we have
\begin{equation*} %\label{eq_funct_h-j}
    \big[\omega_1 \cdot X + Y,\omega_2 \cdot X + Y \big] u(x,t) = 0.
\end{equation*}
\end{lemma}

\begin{proof} Let $(x,t) \in \OT$, and consider the function $v := \log(u)$. Using \eqref{eq_funct_eq-reminder} with $s > t-T$, we obtain
\begin{multline*} %\label{eq_funct_h-j}
    v \!\Big( \!\! \exp(-s (\omega_2 \cdot X\! +\!Y))\! \exp(-s (\omega_1 \cdot X \!+\!Y))
             \exp(s (\omega_2 \cdot X \!+\!Y)) \exp(s (\omega_1 \cdot X \!+\!Y))(x,t)\!\!\Big)\!\! =\\
              s \beta_{\omega_1} + s \beta_{\omega_2} - s \beta_{\omega_1} - s \beta_{\omega_2} +
v(x,t) = v(x,t).
\end{multline*}
We recall the Baker-Campbell-Hausdorff formula
\begin{equation*} %\label{eq_funct_h-j}
\begin{split}	\exp\big(s\tilde Y\big) \exp\big(s\tilde X \big)(x,t) = \exp\left(s
\big(\tilde Y + \tilde X \big) + \tfrac{s^2}{2}\big[ \tilde X, \tilde Y \big] + o(s^2) \right)
\end{split}
\end{equation*}
where $o(s^2)$ denotes a function such that $o(s^2)/s^2 \to 0$ as $s \to 0$, and we apply it twice. The first time we
choose $\tilde X = \omega_1 \cdot X +Y$ and $\tilde Y = \omega_2 \cdot X +Y$, the second time we set $\tilde X =
- \big(\omega_1 \cdot X +Y\big)$ and $\tilde Y = - \big(\omega_2 \cdot X +Y\big)$, and we find
\begin{equation*} %\label{eq_funct_h-j}
    \frac{v \Big( \exp\big(s^2 \big[\omega_1 \cdot X +Y,\omega_2 \cdot X +Y \big] + o(s^2) \big) (x,t)\Big) -
v(x,t)}{s^2}= 0,
\end{equation*}
for every $s > t-T$. Then from the differentiability of the functions $v$ and $\exp$, by letting
$s \to 0$ we obtain
\begin{equation*} %\label{eq_funct_h-j}
    \big[\omega_1 \cdot X +Y,\omega_2 \cdot X +Y \big] v(x,t) = 0.
\end{equation*}
The proof of the claim then follows from the fact that $u(x,t) = \exp \big(v(x,t)\big)$.
\end{proof}

\begin{proof}[Proof of Proposition~\ref{th-comm}]

Let $u: \OT \to \R$ be a nonnegative smooth function, and let $\omega \in \R^m$ be any vector satisfying (H2). We
claim that
\begin{equation} \label{eq_comm-tilde}
    \big[X_k , \omega \cdot X + Y \big] u(x,t) = 0 \qquad k=1, \dots, m,
\end{equation}
for every $(x,t) \in \OT$.

In order to prove \eqref{eq_comm-tilde} we note that, since $\exp\big( s (\omega \cdot X + Y) \big)(x,t) \in
\mathrm{Int}  \left(\A_{(x,t)} \right)(\Omega)$ for any $s \in ]0,s_0[$, there exists $r > 0$ such that
$\exp\big( s \big(\omega \cdot X + r X_k + Y \big) \big)(x,t) \in \mathrm{Int} \left(\A_{(x,t)}(\Omega) \right)$
for $k = 1, \dots, m$. We denote by $e_k$ the $k$-th vector of the canonical basis of $\R^m$, and we apply
Lemma \ref{lem-comm} with $\omega_1 := \omega + r e_k$ and $\omega_2 := \omega$. We find
\begin{equation*} %\label{eq_funct_h-j}
    r \big[X_k , \omega \cdot X + Y \big] u(x,t) = \big[\omega \cdot X + r X_k + Y, \omega \cdot X + Y \big] u(x,t) = 0,
\end{equation*}
for every $(x,t) \in \OT$. This proves \eqref{eq_comm-tilde}.

We apply again Lemma \ref{lem-comm} with $\omega_1 := \omega + r e_k$ and $\omega_2 := \omega + r e_j$, for $j,k = 1,
\dots, m$, and \eqref{eq_comm-tilde} to obtain
\begin{equation*} %\label{eq_funct_h-j}
\begin{split}
     r^2 \big[X_j , X_k \big]u(x,t) & = \big[r X_j + \omega \cdot X + Y, r X_k + \omega \cdot X + Y \big] u(x,t) \\
     & - r \big[X_j, \omega \cdot X + Y \big] u(x,t) + r \big[ X_k, \omega \cdot X + Y \big] u(x,t) = 0.
\end{split}
\end{equation*}
 This proves
\begin{equation} \label{eq_comm-jk}
   \big[X_j , X_k \big] u(x,t) = 0 \qquad j,k=1, \dots, m.
\end{equation}
From \eqref{eq_comm-jk} and from the fact that $\big[X_k, \partial_t \big]=0$, we eventually obtain
\begin{equation*} %\label{eq_funct_h-j}
     \big[X_k, X_0 \big] u(x,t) = \big[X_k, \omega \cdot X + Y \big] u(x,t) - \sum_{j=1}^m \omega_j \big[X_k, X_j
\big] u(x,t) = 0,
\end{equation*}
for $k= 1, \dots, m$. This concludes the proof of \eqref{eq_comm-jk0}. A plain application of the
Baker-Campbell-Hausdorff formula gives the result for all higher-order commutator.

The result for any nonnegative solution then clearly follows from the representation formula \eqref{int_rep}.
\end{proof}

\subsection{Liftable operators} \label{s-ex6}

Our approach applies also to operators that are not invariant with respect to any Lie group structure, but
that can be \emph{lifted} to a suitable operator $\widetilde \L$ that satisfies assumption (H1). Consider, for
instance, the following {Grushin-type evolution operator}
\begin{equation} \label{ex-grushin}
    \L u = \p_{t} u - \p_x^2 u - x^2 \p_y^2 u \qquad (x,y,t) \in \R^3.
\end{equation}
Since it is degenerate at $\big\{ x = 0 \big\}$ and nondegenerate in the set $\big\{ x \ne 0 \big\}$, a change of
variables that preserves the operator cannot exist. If we lift the operator by adding a new variable $w$ and introducing
the vector fields $\widetilde X_1 := X_1$ and $\widetilde X_2 := X_2 + \partial_w$, then we get the lifted operator
\begin{equation} \label{ex-grushin-lift}
    \widetilde \L u: = \p_{t} u - \p_x^2 u - (\p_w + x \p_y)^2 u  \qquad (x,y,w,t) \in \R^4,
\end{equation}
that belongs to the class considered in Section~\ref{sec_Parabolic}. The uniqueness result proved for
\eqref{ex-grushin-lift}
directly extends to \eqref{ex-grushin}.

Analogously, the operator
\begin{equation} \label{ex-CMP-NL}
    \L u = \p_{t} u - \p_x^2 u - x^2 \p_y u \qquad (x,y,t) \in \R^3
\end{equation}
studied in \cite{CintiMenozziPolidoro}, is not invariant with respect to any Lie group structure. However, it can be
lifted to the operator defined in \eqref{ex-CMP} and, also in this case, the uniqueness result for \eqref{ex-CMP} extends to
\eqref{ex-CMP-NL}.
We note that \eqref{ex-CMP-NL} appears in stochastic theory (see the references in \cite{CintiMenozziPolidoro} for a
bibliography on this subject).

Clearly, the lifting method can be applied to a wide class of operators.

\subsection{Open problems}
In this subsection we list several open problems related to the results of the present paper.
\begin{enumerate}
  \item Our first problem concerns with the strict positivity of nonzero nonnegative solutions of the equation $\L
u=0$ in $\OT$ (cf. Theorem~\ref{th-main}).

  \item We would like to  extend our main results to operators with nontrivial zero-order term, namely to operators of
the form
$$\L_cu := \p_t u - \sum_{j=1}^m X_j^2 u - X_0 u - c(x) u.$$

  \item We would like to weaken the left-invariance condition, as well.

  \item We aim to study property {\it (b)} of Section~\ref{ssec_Liouville} for degenerate operators. More precisely,
we would like to find conditions under which the generalized principal eigenvalue $\lambda_0$ of $\L_0$ is equal $0$.
Moreover, we would like to understand whether $\L_0$ is critical in $\R^N$.

  \item In another direction, we would like to extend the nonnegative Liouville-type theorem in $\R^{N+1}$
(Theorem~\ref{thm_end}) to the case of operators with a nontrivial drift term.

  \item Finally, it is natural to extend our work to the case where $\L$ of the form \eqref{e1} is defined on a
noncompact Lie group, and even to the more general setting of a noncompact manifold $M$ with a cocompact group action
(cf. \cite{LinPinchover94}). We expect that the acting group should be nilpotent.
\end{enumerate}

\begin{center}{\bf Acknowledgments} \end{center}
The authors started to work on the present paper during their visit at BCAM - The Basque Center for Applied Mathematics.
The authors wish to thank Professor Enrique Zuazua for the hospitality. The authors thank Professor Nicola
Garofalo for bringing their attention to the Mumford operator. They also wish to thank Alano Ancona for a useful
discussion concerning the Martin representation theorem on Bauer harmonic spaces, and Caterina Manzini for her help in
the study of the Martin functions for Kolmogorov equations.

A.E.~K. and S.~P. are grateful to the Department of Mathematics at the Technion for the hospitality during their visits.
Y.~P. acknowledges the support of the Israel Science Foundation (grants No.~963/11) founded by the Israel Academy of
Sciences and Humanities. The authors also thank Gruppo Nazionale per l'Analisi Matematica, la Probabilit\`a e le
loro Applicazioni (GNAMPA) of the Istituto Nazionale di Alta Matematica (INdAM) for supporting the visit of Y.~P. to Modena
and Reggio Emilia University on February 2014.

\end{document}